\documentclass[10pt]{amsart}
\usepackage{amssymb}
\usepackage{amsmath}
\usepackage{amsthm}
\usepackage{latexsym}
\usepackage{graphicx}
\usepackage{hyperref}
\usepackage{type1cm}
\usepackage{color}
\usepackage{marvosym}

\newcommand{\Z}{{\mathbb Z}}
\newcommand{\R}{{\mathbb R}}
\newcommand{\C}{{\mathbb C}}
\newcommand{\D}{{\mathbb D}}

\newtheorem{lemma}{Lemma}[section]
\newtheorem{theorem}[lemma]{Theorem}
\newtheorem{hypothesis}[lemma]{Hypothesis}
\newtheorem{remark}[lemma]{Remark}
\newtheorem{proposition}[lemma]{Proposition}
\newtheorem{corollary}[lemma]{Corollary}
\newtheorem{definition}[lemma]{Definition}

\newcommand{\nn}{\nonumber}
\newcommand{\be}{\begin{equation}}
\newcommand{\ee}{\end{equation}}
\newcommand{\ul}{\underline}
\newcommand{\ol}{\overline}
\newcommand{\ti}{\tilde}

\newcommand{\spr}[2]{\left\langle #1 , #2 \right\rangle}

\newcommand{\E}{\mathrm{e}}
\newcommand{\I}{\mathrm{i}}

\newcommand{\tr}{\mathrm{tr}}

\DeclareMathOperator{\dist}{dist}

\newcommand{\eps}{\varepsilon}

\newcommand{\lam}{\lambda}

\numberwithin{equation}{section}

\begin{document}

\title[Correlated random operators]{Localization for random operators
 with non-monotone potentials with exponentially decaying correlations}

\author[H.\ Kr\"uger]{Helge Kr\"uger}
\address{Erwin Schr\"odinger Institute, Boltzmanngasse 9, A-1090 Vienna, Austria}
\email{\href{mailto:helge.krueger@rice.edu}{helge.krueger@rice.edu}}
\urladdr{\href{http://math.rice.edu/~hk7/}{http://math.rice.edu/\~{}hk7/}}

\thanks{H.\ K.\ was supported by NSF grant DMS--0800100, a Nettie S. Auttrey followship,
 and an Erwin Schr\"odinger junior research fellowship.}

\date{\today}


\begin{abstract}
 I consider random Schr\"odinger operators with exponentially decaying
 single site potential, which is allowed to change sign. For this
 model, I prove Anderson localization both in the sense of exponentially
 decaying eigenfunctions and dynamical localization. Furthermore,
 the results imply a Wegner-type estimate strong enough to use
 in classical forms of multi-scale analysis.

\end{abstract}

\maketitle

%
%
%

\section{Introduction}

In \cite{ander}, Anderson proposed that randomness of the
potential leads to localization phenomena in the Schr\"odinger
equation. In \cite{fs83}, Fr\"ohlich and Spencer laid the
foundations of {\em multi-scale analysis} to give a mathematical
justification of this phenomenon of {\em Anderson localization}.
Multi-scale analysis was then improved
by a sequence of people notably von Dreifus and Klein in \cite{vDK89} and
Germinet and Klein \cite{gk03}. These forms of multi-scale analysis
were used to prove pure point spectrum with exponentially decaying
eigenfunctions. For recent expositions of multi-scale analysis,
see the book \cite{stollmann} by Stollmann, the lecture notes
\cite{kirsch} by Kirsch, and the review \cite{klein} by Klein.

As pointed out by del Rio, Jitomirskaya, Last, and Simon in \cite{dRJLS}
this is not enough to conclude dynamical properties. This was first
shown using the {\em fractional moments method} by Aizenman and
Molchanov in \cite{am93}. Later methods using multi-scale analysis
were developed by de Bi\'evre and Germinet in \cite{dBG}, Damanik and Stollmann
in \cite{ds}, and Germinet and Klein \cite{gk04}.

All these proofs rely on an explicit a priori bound on the concentration
of eigenvalues  of the Schr\"odinger operator restricted
to a finite box known as a {\em Wegner estimate}. For one-dimensional
Schr\"odinger operators various techniques not relying on a Wegner
estimate exist. This is mainly due to transfer matrices, see for example
Carmona, Klein, and Martinelli \cite{ckm87},
Jitomirskaya \cite{j99}, Bourgain and Goldstein \cite{bg00},
and Damanik, Sims, and Stolz \cite{dss}.
A form of multi-scale analysis for multi-dimensional Schr\"odinger operators
without a Wegner estimate  was first 
developed in the context of quasi-periodic Schr\"odinger operators
by Bourgain, Goldstein, and Schlag in \cite{bgs2} and then improved
by Bourgain in \cite{b2002}, \cite{bbook}, \cite{b2007}.

It was then applied by Bourgain \cite{b2004} and more importantly
Bourgain and Kenig in \cite{bk} to prove
Anderson localization for Bernouilli-Anderson models.
Then used by Germinet, Hislop, and Klein \cite{ghk07}
for the Poisson random potential, and 
by Bourgain in \cite{b2009} to consider certain models with
vector-valued potentials. The methods for Bernouilli potentials
rely on a combinatorial fact known as {\em Sperner's lemma} and
unique continuation properties of the Laplace operator. Both
in \cite{b2009} and this paper, analyticity properties and
smooth distributions will be considered, and a certain fact
about analytic functions, {\em Cartan's lemma}, will play
a key-role. This is similar to the techniques used for
quasi-periodic operators.

\bigskip

My goal in this paper is to continue to develop the methods not
relying on a Wegner estimate. I will consider certain non-monotone
random models with long range correlations. Models with long range
correlations were considered for example by Kirsch, Stollmann, and
Stolz in \cite{kss} and non-monotonous models by Elgart, Tautenhahn,
and Veselic \cite{etv}, Tautenhahn and Veselic \cite{tv09},
and Veselic \cite{v09}.

Let me also point of the work by Baker, Loss, and Stolz in
\cite{bls08}, \cite{bls09} on the random displacement model.
This is a model for continuum Schr\"odinger operators, which
exhibits non-monotonic behavior.

\bigskip

The proof will be much in the spirit of the multi-scale analysis
of von Dreifus and Klein in \cite{vDK89}. However, I will use the
method of Bourgain from \cite{b2009} based on analyticity of
the potential in order to obtain Wegner estimates. Furthermore,
in difference to the results of \cite{bk} and \cite{b2009},
the results of this paper will imply a Wegner estimate, which
is strong enough to start the multi-scale analysis of Damanik
and Stollmann \cite{ds} or Stollmann \cite{stollmann}, which
would imply dynamical localization and pure point spectrum with
exponentially decaying eigenfunctions.

However, I have included a proof of dynamical lozalization at the end of
the paper. This is mainly done to show that the single energy
estimates of this paper are good enough to conclude it.
For a discussion how the single energy estimate
implies exponential decay of the eigenfunctions, I refer the
reader to the work of Bourgain and Kenig \cite{bk}.

%
%
%

\section{Statement of the results}
\label{sec:results}

Before introducing the general assumptions on the potential
$V$, I will introduce the class of {\em alloy-type potentials}
which will serve as an example.
Let the {\em single-site potential} $\varphi: \Z^d \to \R$ satisfy
\be
 \varphi(0) \neq 0,\quad |\varphi(n)| \leq \E^{-c |n|_{\infty}}
\ee
for some positive constant $c > 0$ and $|n|_{\infty} = \max_{1 \leq j \leq d} |n_j|$.
Furthermore let $\{\omega_x\}_{x \in \Z^d}$ be independent and identically
distributed random variables in $[-\frac{1}{2}, \frac{1}{2}]$,
whose density $\rho$ is bounded.
For $\lam > 0$ and $\omega$, we introduce the {\em potential}
\be\label{eq:potalloy}
 V_{\lam, \omega}(x) = \lam\left(\sum_{m \in \Z^d} \omega_{x + m} \varphi(m)\right).
\ee
Then our Schr\"odinger operator $H_{\lam,\omega}: \ell^2(\Z^d)\to \ell^2(\Z^d)$
is defined by
\be
 H_{\lam, \omega} = \Delta + V_{\lam, \omega},
\ee
where $\Delta u (x) = \sum_{|e|_1 = 1} u(x+e)$ is the discrete Laplacian
with $|n|_1 = \sum_{k=1}^{d} |n_k|$.
Operators of this form have been studied for example in 
\cite{etv}, \cite{tv09}, and \cite{v09}.

\bigskip

Before, defining all the properties of $H_{\lam,\omega}$, we wil
study, I will state the main results in the case of an alloy-type
potential.

\begin{theorem}\label{thm:int1}
 Let $\lam > 0$ be large enough, then we have
 \begin{enumerate}
  \item There exists an interval $\Sigma_{\lambda}$ such that for
   almost every $\omega$
   \be
    \sigma(H_{\lam,\omega}) = \Sigma_{\lambda}.
   \ee
  \item For almost every $\omega$, $H_{\lam,\omega}$ exhibits
   Anderson localization (see Definition~\ref{def:anderloc}).
  \item For almost every $\omega$, $H_{\lam,\omega}$ exhibits
   dynamical localization (see Definition~\ref{def:dynloc}).
  \item For $\beta > 1$, there exists a constant $C_{\beta} > 0$
   such that the integrated density of states $\mathcal{N}(E)$
   (see Definition~\ref{def:ids}) obeys
   \be\label{eq:contids}
    \mathcal{N}(E + \eps) - \mathcal{N}(E - \eps)
     \leq \frac{C_{\beta}}{\log(\eps^{-1})^{\beta}}
   \ee
   for any $E \in \R$ and $\eps \in (0, \frac{1}{2})$.
 \end{enumerate}
\end{theorem}

(i) holds for all $\lam > 0$ and will be proven in Appendix~\ref{sec:spectrum}.
A finite volume version of (iv) holds, see Theorem~\ref{thm:intB}.
The continuity of the integrated density given here is probably
not optimal. In fact Veseli\'c has shown that if the density
$\rho$ is of bounded variation and $\sum_{n\in\Z^d} \varphi(n) \neq 0$
that \eqref{eq:contids} can be improved to
\be
 \mathcal{N}(E+\eps)-\mathcal{N}(E-\eps) \leq C \eps
\ee
for some constant $C > 0$. That is the integrated density
of states is Lipschitz continuous. This can be found in \cite{v09},
which I recommend also for discussions of earlier results
on the continuity of the integrated density of states
for models with alloy-type potential.
I will now define the two localization properties from
Theorem~\ref{thm:int1}.

\begin{definition}\label{def:anderloc}
 The operator $H: \ell^2(\Z^d) \to \ell^2(\Z^d)$ is said to 
 exhibit Anderson localization, if its spectrum is pure point
 and there exists a constant $\gamma > 0$ such that
 for every eigenfunction $\psi$, we have for $n\in\Z^d$ that
 \be
  |\psi(n)| \leq C_{\psi} \E^{-\gamma |n|_{\infty}},
 \ee
 where $C_{\psi} > 0$ is a constant.
\end{definition}

I denote by $\{e_x\}_{x\in\Z^d}$ the standard basis of $\ell^2(\Z^d)$
given by
\be
 e_x(n)=\begin{cases} 1, &x=n;\\ 0,& \text{otherwise}.\end{cases}
\ee
The function $\psi(t) = \E^{-\I t H} e_x$ is the solution
of the problem
\be\begin{split}
 \psi(0) &= e_x \\
 \I \partial_t \psi(t) &= H \psi(t),
\end{split}\ee
which is known as Schr\"odinger's equation. The second
localization property, we are interested in is

\begin{definition}\label{def:dynloc}
 The operator $H: \ell^2(\Z^d) \to \ell^2(\Z^d)$ is said
 to exhibit dynamical localization, if for every $x \in \Z^d$
 and $p \geq 1$ we have
 \be
  \sup_{t\in\R} \left(\sum_{n \in \Z^d} (1 + |n|_{\infty})^{p}
   |\spr{e_n}{ \E^{-\I tH} e_x}|^2 \right) < \infty.
 \ee
\end{definition}

I note that this is not a very strong localization
property. I believe that it is possible to show stronger
localization properties than this, but have decided not
to do so to keep this work at a reasonable length.

We now begin by introducing the {\em integrated
density of states}.
Denote by $\Lambda_r(x)$ the cube with radius $r$
and center $x$
\be
 \Lambda_r(x) = \{n \in \Z^d:\quad |n - x|_{\infty} \leq r\}.
\ee
$H_{\lam,\omega}^{\Lambda_r(x)}$ denotes the restriction
of $H_{\lam,\omega}$ to $\ell^2(\Lambda_r(x))$. We denote
the number of eigenvalues of $H_{\lam,\omega}^{\Lambda_r(x)}$
in the interval $[E_0,E_1]$ by
\be
 \tr(P_{[E_0, E_1]}(H_{\lam,\omega}^{\Lambda_r(x)})).
\ee
Furthermore, $\mathbb{E}$ denotes the expectation value,
and $\#(\Lambda_r(0))$ denotes the number of elements of $\Lambda_r(0)$.

\begin{definition}\label{def:ids}
 The integrated density of states $\mathcal{N}(E)$ is given by
 \be
  \mathcal{N}(E) = \lim_{r \to \infty}
   \left(\frac{1}{\#(\Lambda_r(0))}  
    \tr(P_{(-\infty, E]}(H_{\lam,\omega}^{\Lambda_r(0)}))
     \right).
 \ee
\end{definition}

Here the limit is known to exist (see Section~5 of \cite{kirsch}).
We will now discuss classes of potential, which include
potentials of the form \eqref{eq:potalloy}, but which
isolate the properties needed for the proof.

\begin{hypothesis}\label{hyp:expdecay}
 The potential $V_{\lam,\omega}(x)$ is said to have exponentially
 decaying correlations, if the following properties hold.
 \begin{enumerate}
  \item There exists a map
   \be
    f: \left[-\frac{1}{2},\frac{1}{2}\right]^{\Z^d} \to \R.
   \ee
   such that $V_{\lam,\omega}(x) = \lam f(T_x \omega)$, where $(T_x\omega)_n = \omega_{x+n}$.
  \item There exists $c > 0$, such that if
   $\omega_n = \ti{\omega}_n$ for $n \in \Lambda_r(0)$, we have
   \be\label{eq:asfr2}
    |f(\omega) - f(\ti{\omega})| \leq \E^{- c r}.
   \ee
  \item There are constants $F > 0$ and $\alpha > 0$ such that
   for any $E \in \R$ and $\eps > 0$, we have
   \be  
    \mu(\{\omega:\quad |f(\omega) - E| \leq \eps\})\leq F \cdot \eps^{\alpha}.
   \ee
  \item $\mathbb{P}$ is the product measure of a fixed probability
   measure $\mu$ on $\left[-\frac{1}{2}, \frac{1}{2}\right]$.
 \end{enumerate}
\end{hypothesis}

In order to see that the potential defined in \eqref{eq:potalloy}
satisfies this hypothesis, define
$$
 f(\omega) = \sum_{m\in\Z^d} \omega_m \varphi(m).
$$
Then (i) holds, (ii) follows from $|\varphi(n)| \leq \E^{-c|n|_{\infty}}$,
and (iii) from $\varphi(0) \neq 0$.

Assumptions (i) and (ii) imply that properties of $H_{\lam,\omega}^{\Lambda_r(x)}$
and $H_{\lam,\omega}^{\Lambda_r(y)}$ become almost independent
if $|x - y|_{\infty}$ is large enough. For example Lemma~\ref{lem:propXn}
is an implementation of this fact.
Assumption (iii) is necessary to obtain an initial condition for
multi-scale analysis, see Appendix~\ref{sec:initcond}. 

It is furthermore noteworthy that 
Conditions (i) and (ii) imply that the
function $f$ is H\"older continuous, if we use the metric
\be
 d(\omega, \ti{\omega}) = \sum_{n \in \Z^d} \frac{|\omega - \ti{\omega}|}{2^{|n|_{\infty}}}
\ee
on $[-\frac{1}{2},\frac{1}{2}]^{\Z^d}$. This is the natural
metric, since under it also the maps $T_x$ are H\"older
continuous.

The following theorem illustrates that Hypothesis~\ref{hyp:expdecay}
combined with a {\em Wegner estimate} is already sufficient to
prove Anderson localization. 

\begin{theorem}\label{thm:intwithwegner}
 Assume Hypothesis~\ref{hyp:expdecay}.
 In addition assume the Wegner estimate
 for $R \geq 1$ large enough and all $\eps > 0$
 \be\label{eq:asswegner}
  \mathbb{E}(\tr(P_{[E-\eps, E+\eps]}(H_{\lam,\omega}^{\Lambda_R(0)}))) 
   \leq \frac{C^{W} (\#\Lambda_R(0))^{b}}{\log(\eps^{-1})^{2 ( b \cdot d + \beta)}},
 \ee
 where $b \geq 0$, $C^{W}$ is a $\lam$ independent constant, and $\beta > 0$
 is large enough.

 Then for $\lam > 0$ large enough and almost every $\omega$ Anderson localization 
 and dynamical localization hold.
\end{theorem}

I have decided to include this theorem,
since it is a good dividing line to lay out the framework
of the classical parts of multi-scale analysis, which can
also be found for example in the already mentioned works
by Kirsch \cite{kirsch} and Stollmann \cite{stollmann}.
The main difference between the proof of this theorem,
and with what we will use later, is that for this theorem
we only need to allow one bad cube, and later we will let
this number go to infinity. The tools for this proof
are developed in Section~\ref{sec:suitability} to
\ref{sec:combinatoric1}. The proof is then given in
Section~\ref{sec:asswegner}.

It is still possible in Theorem~\ref{thm:intwithwegner}
that the measure $\mu$ is a Bernouilli
measure. However, it is unclear how to prove a Wegner estimate
in this generality. Except in special cases like the one
treated by Bourgain in \cite{b2004}. See also
Veseli\'c \cite{v09} for Wegner estimates for 
$\varphi$ finitely supported and $\mu$ absolutely continuous.
In fact, we will make an analyticity assumption
on the potential in Hypothesis~\ref{hyp:analytic}, which will
serve as a replacement of \eqref{eq:asswegner}. In the case of the potential defined in
\eqref{eq:potalloy} the assumption reduces to the measure $\mu$
being absolutely continuous with a bounded density.

We denote by $\mathbb{D}$ the disk in $\C$ with center $0$ and radius
$6$, that is
\be
 \mathbb{D} = \{z \in \C:\quad |z| < 6\}.
\ee
We are now ready for

\begin{hypothesis}\label{hyp:analytic}
 We say the $V_{\lam,\omega}(x)$ obeys the analyticity assumption,
 if the following hold.
 \begin{enumerate}
  \item There exist a sequence of maps
   \be
    f_r:\mathbb{D}^{\Lambda_r(0)} \to \C.
   \ee
   and a constant $c > 0$ satisfying the tail estimate for $R \geq 0$
   \be\label{eq:asfr}
    \sum_{r = R}^{\infty} \|f_r\|_{L^{\infty}(\mathbb{D}^{\Lambda_r(0)})}
    \leq \E^{-c R}.  
   \ee
  \item The potential is given by
   \be\label{eq:defVn}
    V_{\lam,\omega}(x) = 
     \lam \left(\sum_{r=0}^{\infty} f_{r}(\{\omega_{n}\}_{n \in \Lambda_r(x)})\right).
   \ee
  \item The map
   \be
    \omega_0 \mapsto V_{\lam,\omega}(0) 
   \ee
   is non-constant for any choice of $\{\omega_x\}_{x\neq 0}$.
  \item The measure $\mu$ has a density $\rho$, which is bounded.
 \end{enumerate}
\end{hypothesis}

One also can check that the potential defined in \eqref{eq:potalloy}
satisfies this hypothesis. It is also possible to show that
Hypothesis~\ref{hyp:analytic} implies Hypothesis~\ref{hyp:expdecay}.
A proof that (iii) of Hypothesis~\ref{hyp:expdecay} holds is
given in Lemma~\ref{lem:loja}.
The main result of this paper is

\begin{theorem}\label{thm:intA}
 Assume Hypothesis~\ref{hyp:analytic}. 
 Then for $\lam > 0$ large enough, Anderson localization and
 dynamical localization hold for almost every $\omega$.
\end{theorem}

We also have

\begin{theorem}\label{thm:intB}
 Assume Hypothesis~\ref{hyp:analytic} and $\lambda > 0$ large enough. 
 For any $\beta \geq 1$ there exists a constant $C_{\beta}$
 and a length scale $R_{\beta}$ such that for all $R \geq R_{\beta}$
 the Wegner estimate
 \be
  \mathbb{E}\left(\frac{1}{\#\Lambda_R(0)}\tr(P_{[E-\eps, E+\eps]}(H_{\lam,\omega}^{\Lambda_R(0)}))\right) 
   \leq \frac{C_{\beta}}{\log(\eps^{-1})^{\beta}}
 \ee
 holds for $E\in\R$ and $\eps\in (0,\frac{1}{2})$.
\end{theorem}

These two theorems imply Theorem~\ref{thm:int1}.
Actually, the results imply an estimate of the form
\be
 \mathbb{E}\left(\frac{1}{\#\Lambda_R(0)}\tr(P_{[E-\eps, E+\eps]}(H_{\lam,\omega}^{\Lambda_R(0)}))\right) 
   \leq \E^{-C \log(\log(\frac{1}{\eps}))^{3/2}}
\ee
for some $C > 0$. I have decided to state it in the form
of Theorem~\ref{thm:intB}, since it is somewhat easier
on the eye.

The proof of Theorem~\ref{thm:intA} and \ref{thm:intB}
proceed by a version of multi-scale analysis similar
to the one of Bourgain in \cite{b2009}. Let me point
out some new features except that I allow for long
range correlations. The control on the probabilities is super polynomial
in the length scale as in the work of Germinet and Klein
\cite{gk01}, which allows us to proof Theorem~\ref{thm:intB}.
Instead of working with {\em elementary regions} as in \cite{b2009},
I work only with boxes $\Lambda_r(x)$ in the proof.

At this point, I also wish to point out that the
result should be easy to extend in various direction.
First, one should be able to allow arbitrary 
background operators and not just the Laplacian $\Delta$.
Second, one should be able to extend the proof to more
general underlying probability spaces then $[-\frac{1}{2},\frac{1}{2}]$
and vector valued potentials as done in \cite{b2009}.
The essential point here is that one still need an
analog of Cartan's Lemma to hold.
Third, one should be able to use the methods to
understand localization in Lifschitz tails.
I will give some more comments on further directions
in the next section.

\bigskip

Let me now discuss the example that motivated Hypothesis~\ref{hyp:analytic}.
Consider the times two map
\be\begin{split}
 T: [0,1]&\to [0,1] \\
   Tx&=2 x \pmod{1}.
\end{split}\ee
It is relatively well known that this is an ergodic transformation.
Furthermore, this transformation has attracted some attention in
spectral theory \cite{bs00}, \cite{chsp}.
Using the binary expansion $x = \sum_{i=1}^{\infty} \frac{x_i}{2^{i}}$,
we can conjugate $T$ to the shift on the space $[0,1]^{\Z_+}$ with 
respect to the Bernouilli measure. 

One can show that for an analytic and one-periodic function $g$,
one has that
\be
 f(\omega) = g\left(\sum_{j=1}^{\infty} \frac{1}{2^{j}} \omega_j\right)
\ee
satisfies Hypothesis~\ref{hyp:analytic}. So transformations like
the doubling map would be in our framework, if we could relax
the absolutely continuous measure to a pure point one.

\bigskip

In Section~\ref{sec:multischeme}, we show how the procedure of multi-scale
analysis works. As already mentioned, we lay out the basics
for the multi-scale analysis with a Wegner estimate in
Sections~\ref{sec:suitability} to \ref{sec:combinatoric1}.
Then we provide the proof of Theorem~\ref{thm:intwithwegner}
in Section~\ref{sec:asswegner}.

Sections~\ref{sec:resolv1} to \ref{sec:cartanschroe} discuss the machinery used
to replace the Wegner estimate. The main result is
Theorem~\ref{thm:schroecartan}, which essentially proofs
a Wegner type estimate on a large scale from a weak
assumption at the large scale and strong assumptions on
a smaller scale. The main ingredient here is {\em Cartan's Lemma},
whose different forms we review in Section~\ref{sec:cartan}.
Sections~\ref{sec:combinatoric2} to \ref{sec:proofthmindstep}
then contain the main elements of the proof of Theorem~\ref{thm:intA}.

In Sections~\ref{sec:firststeplocalization} to \ref{sec:dynloc},
we prove that the conclusions of our multi-scale analysis
imply dynamical localization. I have decided not to include
a proof of Anderson localization, since it is similar to the
argument of Bourgain and Kenig \cite{bk}. Also it follows
by running a multi-scale analysis as in Kirsch \cite{kirsch}
having the Wegner estimate from Theorem~\ref{thm:intB}
at ones disposition.

In Appendix~\ref{sec:initcond}, I demonstrate how one
can deduce the initial condition for multi-scale analysis,
essentially from a largeness assumption on $\lambda$.
In Appendix~\ref{sec:spectrum}, I have included an argument
that shows that the spectrum of the operators under consideration
is an interval. Lastly, Appendix~\ref{sec:wegner} demonstrates
how the conclusions of multi-scale analysis imply a
Wegner estimate.

%
%
%

\section{Comments, Improvements, and Questions}

In this section, I want to discuss some further directions
the results of this paper can be improved on. I also wish to
ask some questions.

It is required in Hypothesis~\ref{hyp:analytic} that
all the functions $f_r$ are defined on $\mathbb{D}^{\Lambda_r(0)}$,
where
\be
 \D = \{z:\quad |z| < 6\}.
\ee
It would be natural to replace $\D$ in this definition by
\be
 \mathcal{A}_{\rho} = \{z:\quad \dist(z, [-\frac{1}{2}, \frac{1}{2}]) < \rho\},
\ee
a $\rho$ neighborhood of $[-\frac{1}{2}, \frac{1}{2}]$.
Then at least two changes are necessary first, one needs
to add a covering argument to the application of Cartan's
lemma (e.g. in the proof of Theorem~\ref{thm:schroecartan}).
Then the assumption (v) of it changes to the existence of 
$x_0 \in [-\frac{1}{2}, \frac{1}{2}]^n$ to the existence
of such an $x_0$ in every ball of radius $\frac{\rho}{12}$.
I believe it is possible to change the probabilistic
arguments to show this. Such a treatment is necessary
for quasi-periodic systems and ones defined by the
skew-shift, see \cite{bbook}.

It should be also be possible to replace $[-\frac{1}{2},\frac{1}{2}]$
by more general sets. An interesting example are Lie groups
such as $SU(n)$ as discussed by Bourgain in \cite{b2009}.
I believe that probably changes as above will also be necessary
in the proof.

\bigskip

Reducing the size of the sets, where the functions $f_r$
are analytic is a way of lowering the regularity of these
functions. Similarly, one could ask if the result stays true
for smooth functions or quasi-analytic functions.
By the remarks following Hypothesis~\ref{hyp:expdecay},
we know that our assumptions imply that the function
$f$ is H\"older continuous on $[-\frac{1}{2}, \frac{1}{2}]^{\Z^d}$.
It is also intriguing if the assumption of exponential
decay of the correlations is optimal. I would expect
that one can relax condition (ii) of Hypothesis~\ref{hyp:expdecay}
significantly.

The main motivation for this is that Kirsch, Stollmann,
and Stolz have shown in \cite{kss} that an assumption of
the form ($\eps > 0$ and $C >0$)
\be
 |f(\omega) - f(\ul{\omega})| \leq \frac{C}{r^{2d + \eps}}
\ee
is sufficient to carry out multi-scale analysis if
a Wegner estimate is available. It is known from Veseli\'c work \cite{v09}
that such an estimate holds for the alloy-type potential as
long as $\sum_{n\in\Z^d} \varphi(n) \neq 0$.

\bigskip

This paper shows localization at large disorders. The usual
proofs of localization also show it near the band-edges. It
would be interesting to obtain such a result for the potentials
considered here. The necessary improvement is not in the results
of this paper, but to prove {\em Lifschitz tails} to have
an initial scale estimate for multi-scale analysis.

\bigskip

As mentioned after Theorem~\ref{thm:int1}, I do not expect
that control of the integrated density of states in this
paper to be optimal. One should probably at least expect
some form of H\"older continuity of it. See also the
paper \cite{v09} of Veseli\'c for further discussion and
currently the best results in this direction.

%
%
%

\section{The multi-scale scheme}
\label{sec:multischeme}

The goal of this section is to discuss the main aspects of the proofs
of the theorems from the introduction. As the classical
multi-scale analysis, we will be concerned with showing that certain
estimates on the inverses of the restrictions of $H_{\lam,\omega}$
to boxes $\Lambda_r(0)$ hold with large probability. We make the
required properties precise in the following definition.

\begin{definition}\label{def:acceptable}
 An interval $[r_0, r_1]$ is called $(\gamma,\alpha)$-acceptable for $H_{\lam,\omega} - E$
 if for $r_0 \leq r \leq r_1$ there exists an event $\mathcal{B}_r$ such that
 \be
  \mathbb{P}(\mathcal{B}_r) \leq \frac{1}{r^{\alpha}}
 \ee
 and for $\omega \notin \mathcal{B}_r$ we have
 \begin{enumerate}
  \item The resolvent estimate
   \be
    \|(H_{\lam,\omega}^{\Lambda_r(0)} - E)^{-1}\| \leq \frac{1}{8}\E^{ \sqrt{r}}.
   \ee
  \item For any pair $x,y \in \Lambda_r(0)$ with $|x - y| \geq \frac{r}{10}$
   \be
    |\spr{e_x}{(H_{\lam,\omega}^{\Lambda_r(0)} - E)^{-1} e_y}| 
     \leq \frac{1}{8} \frac{1}{\#(\partial_-(\Lambda_r(0)))} \E^{-\gamma |x-y|}.
   \ee
 \end{enumerate}
\end{definition}

Here $\#(\Xi)$ denotes the number of elements of $\Xi\subseteq\Z^d$
and $\partial_-(\Lambda_r(0))$ the inner boundary defined
in \eqref{eq:definnbdd}. We will begin by stating the initial
condition of multi-scale analysis. The proof uses the largeness
assumption on $\lam > 0$ and is given in Appendix~\ref{sec:initcond}.

\begin{proposition}\label{prop:msinitcond}
 Assume Hypothesis~\ref{hyp:expdecay}.
 Let $r\geq 1$ and $\alpha > 0$.
 Then there exists $\lam_0 = \lam_0(r,\alpha)$ such that
 for any $E \in \R$ and $\lam \geq \lam_0$.
 \be
  [1, r]\text{ is $(2, \alpha)$-acceptable for } H_{\lam,\omega} - E.
 \ee
\end{proposition}

The proof of this is pretty standard, and only uses condition
(iii) of Hypothesis~\ref{hyp:expdecay}. We now come to results
that allow us to extend the range of the interval, that
is $(\gamma,\alpha)$-acceptable.

\begin{theorem}\label{thm:indstepwegner}
 Assume Hypothesis~\ref{hyp:expdecay} and \eqref{eq:asswegner}
 for $\beta > 1$ large enough.
 Let $1 \leq \gamma \leq 2$, $r \geq 1$ large enough,
 and 
 \be
  \alpha = d \frac{2 d+ 1}{2d - 1}.
 \ee
 Assume 
 \be
  \{r\} \text{ is $(\gamma,\alpha)$-acceptable for }H_{\lam,\omega} - E.
 \ee
 Then
 \be
  [R_0, R_1]\text{ is $(\ti{\gamma}, \alpha)$-acceptable for }H_{\lam,\omega} - E,
 \ee
 where 
 \be
  R_0 = \left\lceil (r)^{1 + \frac{1}{5 d}} \right\rceil,\quad
  R_1 = \left\lfloor (r)^{1+\frac{1}{2d}} \right\rfloor,\quad 
  \ti{\gamma} \geq \gamma \left(1 - \frac{200}{r^{\frac{1}{d}}}\right).
 \ee
\end{theorem}

The proof of this statement is well known, and can for example be found
in the lecture notes \cite{kirsch} of Kirsch or in the paper
by von Dreifus and Klein \cite{vDK89}. We will give a proof
in Sections~\ref{sec:suitability} to \ref{sec:asswegner}.
Let us now record the main consequence

\begin{corollary}
 Assume Hypothesis~\ref{hyp:expdecay} and \eqref{eq:asswegner}
 for $\beta > 1$ large enough.
 Let $\alpha = d \frac{2d+1}{2d -1}$. Then
 for $\lam \geq \lam_0$ and all energies $E$
 \be
  [1,\infty)\text{ is $(1,\alpha)$-acceptable for } H_{\lam,\omega} - E
 \ee
\end{corollary}

\begin{proof}
 One needs to check that $\ti{\gamma}$ defined in Theorem~\ref{thm:indstepwegner}
 stays $\geq 1$ while applying Theorem~\ref{thm:indstepwegner}
 countably many times. This can be done as in Corollary~\ref{cor:msa2}.
\end{proof}

Before deriving consequences of this corollary,
in particular proving Theorem~\ref{thm:intwithwegner},
we will look at the multi-scale steps of our second
approach.

\begin{theorem}\label{thm:indstep}
 Assume Hypothesis~\ref{hyp:analytic} and that $r$ is large enough.

 Assume for $\alpha \geq 3d + 2$ that
 \be
  [r, r^3]\text{ is $(\gamma,\alpha)$-acceptable for }H_{\lam,\omega} - E.
 \ee
 Then with 
 \be
  \ti{\alpha}_{r} = \left\lfloor\sqrt{\frac{1}{2d} \frac{\log(r)}{\log(\max(4,2 + \frac{4\gamma}{c}))}}\right\rfloor
 \ee
 we have
 \be
  [R_0, R_1]\text{ is $(\ti{\gamma}, \ti{\alpha}_{r})$-acceptable for }H_{\lam,\omega} - E,
 \ee
 where 
 \be
  R_0 = r^{3d + 8},\quad R_1 = r^{\frac{\alpha - 1}{d+ 1}},\quad
  \ti{\gamma} = \gamma \cdot \left(1-\frac{2}{r}\right).
 \ee
\end{theorem}

The proof of this theorem will be given in Section~\ref{sec:proofthmindstep}.
This theorem will allow us to prove

\begin{corollary}\label{cor:msa2}
 For any $\beta$, there exists $R_{\beta}$ such that
 \be\label{eq:acceptforallbeta}
  [R_{\beta}, \infty)\text{ is $(1, \beta)$-acceptable for } H_{\lam,\omega} - E.
 \ee
\end{corollary}

\begin{proof}
 Fix $E\in \R$.
 By Proposition~\ref{prop:msinitcond}, we can conclude that for some large
 enough $r$ and $\alpha \geq 1$, we have
 $$
  [1, r^{3}]\text{ is $(2,\alpha)$-acceptable for }H_{\lam,\omega} -E.
 $$
 We choose $\alpha$ such that $\frac{\alpha-1}{d+1} \geq (3d+8)^3$,
 and $r$ so large that $\ti{\alpha}_r \geq \alpha$. Then, we 
 define a sequence of $R_k$
 $$
  R_1 = r^{3d + 8},\quad R_k = (R_{k-1})^{3d + 8},
 $$
 satisfying $R_k \geq r^{k}$, and
 $$
  \gamma_1 = 2,\quad \gamma_k = \gamma_{k-1}\left(1 - \frac{2}{R_{k-1}}\right).
 $$
 We obtain by Theorem~\ref{thm:indstep} that
 $$
  [R_k, (R_k)^{3d+9}]\text{ is $(\gamma_k, \alpha_{R_k})$-acceptable for } H_{\lam,\omega} -E.
 $$
 In particular $\alpha_{R_k} \to \infty$ as $k \to \infty$.
 It remains to check that $\gamma_k \geq 1$, but
 this is easy. The claim follows.
\end{proof}

Before discussing the results on localization,
let us quickly prove the result on the
Wegner estimate.

\begin{proof}[Proof of Theorem~\ref{thm:intB}]
 This follows from the previous corollary combined with
 the results from Section~\ref{sec:wegner}.
\end{proof}

We now come to

\begin{theorem}\label{thm:dynlocmsa}
 Let $r_{\infty} \geq 1$, $\gamma > 0$.
 Assume Hypothesis~\ref{hyp:expdecay} and
 \be
  [r_{\infty},\infty)\text{ is $(\gamma, 4d)$-acceptable for } H_{\lam,\omega} - E.
 \ee
 Then $H_{\lam,\omega} - E$ exhibits dynamical localization
 for almost every $\omega$.
\end{theorem}

We can obtain Anderson localization as in \cite{bk}.
We see that Theorem~\ref{thm:intwithwegner}
and Theorem~\ref{thm:intA} follow.

%
%
%

\section{Suitability and probabilistic estimates}
\label{sec:suitability}

In this section, we first introduce the notion of suitability,
which quantifies the properties of restrictions to finite boxes.
In order to emphasize that these notions are independent of
the specific setting, we will write $H$ instead of $H_{\lam,\omega}$.
Second, we will derive the main probabilistic estimates needed
for the proof of Theorem~\ref{thm:indstep}.

We denote by $\{e_x\}_{x \in \Z^d}$ the standard basis of
$\ell^2(\Z^d)$, that is
\be
 e_x(n) = \begin{cases} 1,& n = x; \\ 0,& \text{otherwise}.\end{cases}
\ee
We recall that for $\Lambda_r(x) = \{n\in\Z^d:\ |x-n|_{\infty}\leq r\}$
a box in $\Z^d$, we denote by
$H^{\Lambda_r(x)}$ the restriction of $H$ to $\ell^2(\Lambda_r(x))$.
We will denote by $\partial_- \Lambda_r(x)$ the {\em inner boundary},
that is
\be\label{eq:definnbdd}
 \partial_- \Lambda_r(x) = \{n \in \Lambda_r(x):\quad \exists m \in \Z^d \setminus \Lambda_r(x):\ |n-m|_1 = 1\}.
\ee
We note that $\partial_-\Lambda_r(x) = \Lambda_r(x) \setminus \Lambda_{r-1}(x)$.
In particular,
\be
 \# (\partial_- \Lambda_r(x)) \leq 2d (3r)^{d-1}.
\ee

\begin{definition}\label{def:suitable}
 Let $\gamma > 0$, $\tau \in (0,1)$, and $p \geq 0$ an integer.
 A box $\Lambda_r(n)$ is called $(\gamma,\tau,p)$-suitable for $H - E$
 if the following hold.
 \begin{enumerate}
  \item The resolvent estimate
   \be
    \|(H^{\Lambda_r(n)} - E)^{-1}\| \leq \frac{1}{2^p}\E^{ r^{\tau}}.
   \ee
  \item For $x,y \in \Lambda_r(n)$ satisfying $|x -y| \geq \frac{r}{10}$
   \be
    |\spr{e_x}{(H^{\Lambda_r(n)} - E)^{-1} e_y}| \leq 
     \frac{1}{2^{p}} \frac{1}{\#(\partial_-\Lambda_r(x))} \E^{-\gamma |x-y|}.
   \ee
 \end{enumerate}
\end{definition}

This definition is made in such a way, that it becomes useful
for the computations in Sections~\ref{sec:green} and \ref{sec:resolv1}.
Furthermore, this definition is more general than we will need. In fact, it will
suffice for the purposes of this paper to restrict to the
case $\tau = \frac{1}{2}$ and $p \in \{0,1,2,3\}$. The main motivation
for the more general definition is that parameters $\tau$ close
to $1$ are necessary, if one wants to treat skew-shifts on
tori of large dimension, see \cite{kphd}. Let me make the connection
to $(\gamma, \alpha)$-acceptable precise:

\begin{remark}
 We have that $[r_0, r_1]$ is $(\gamma,\alpha)$-acceptable in the sense
 of Definition~\ref{def:acceptable} if and only if
 for $r \in [r_0, r_1]$ we have
 \be
  \mathbb{P}(\Lambda_r(0)\text{ is not $(\gamma,\frac{1}{2},3)$-suitable for }H_{\lam,\omega} - E)
   \leq \frac{1}{r^{\alpha}}.
 \ee
\end{remark}

The integer $p$ will serve as a parameter, we can lower if
we need to perturb $E$ or $H$ slightly. We will now
proceed to make this precise.

\begin{lemma}\label{lem:suitstable}
 Let $\gamma > 0$, $\tau \in (0,1)$, and $p \geq 1$ an integer.
 Assume $\Lambda_r(n)$ is $(\gamma,\tau,p)$-suitable for $H - E$
 and
 \be
  \|(\ti{H}^{\Lambda_r(n)} - \ti{E}) - (H^{\Lambda_r(n)} - E)\| 
   \leq \frac{1}{2^{p+1}} \frac{1}{\#(\partial_-\Lambda_r(x))} \E^{- \gamma r  - 2 r^{\tau}}.
 \ee
 Then $\Lambda_r(n)$ is $(\gamma,\tau,p-1)$-suitable for $\ti{H} - \ti{E}$.
\end{lemma}

\begin{proof}
 Denote $A = H^{\Lambda_r(n)} - E$ and $B = \ti{H}^{\Lambda_r(n)} - \ti{E}$.
 From $B^{-1} - A^{-1} = B^{-1} (B - A) A^{-1}$,
 we obtain
 $$
  B^{-1} = A^{-1} (\mathbb{I} + (A - B) A^{-1})^{-1}
 $$
 By assumption, we have $\|(A - B) A^{-1}\| \leq \frac{1}{2}$ and thus
 \begin{align*}
  \|B^{-1}\| & \leq 2 \|A^{-1}\|, \\
  |\spr{e_x}{B^{-1} e_y}| & \leq |\spr{e_x}{A^{-1} e_y}| + 2 \|A^{-1}\|^2 \|B - A\|.
 \end{align*}
 The claim follows.
\end{proof}

We now specialize to the case of estimating a perturbation in the energy

\begin{lemma}\label{lem:perturbEsuitable}
 Assume
 \be\label{eq:asrgamp}
  d 2^{p+2} (3 r)^{d- 1} \leq \E^{\gamma r},\quad \gamma r^{1 - \tau} \geq 1,
 \ee
 that $\Lambda_r(x)$ is $(\gamma, \tau, p)$-suitable for $H - E$,
 and that 
 \be
  |E - \hat{E}| \leq \E^{-4 \gamma r}.
 \ee
 Then $\Lambda_r(x)$ is $(\gamma,\tau,p-1)$-suitable for $H - \hat{E}$.
\end{lemma}

\begin{proof}
 \eqref{eq:asrgamp} implies
 $$
  2^{p+1} \#(\partial_{-}\Lambda_r(x)) \E^{2 r^{\tau}} \leq \E^{3 \gamma r}.
 $$
 Thus the claim follows from Lemma~\ref{lem:suitstable}.
\end{proof}

We will now begin to study properties specific to potentials
obeying either Hypothesis~\ref{hyp:expdecay} or \ref{hyp:analytic}.
In particular \eqref{eq:asfr2} respectively \eqref{eq:asfr}
will be important. The first goal will be to understand
what happens if we change $\omega$. For $\Xi \subseteq \Z^d$,
we denote by $\Xi^c = \Z^d\setminus \Xi$ its complement.

\begin{definition}
 Let $\omega, \ti{\omega} \in \Omega$ and $\Xi \subseteq \Z^d$, we will write
 $\omega = \ti{\omega} \pmod{\Xi}$
 if
 \be
  \omega_x = \ti{\omega}_x,\quad x \in \Xi^c.
 \ee
\end{definition}

Having this definition, we are ready for

\begin{lemma}
 Let $\omega = \ti{\omega} \pmod{\Lambda_r(x) ^c}$, then
 \be
  |V_{\lam, \omega}(x) - V_{\lam, \ti{\omega}}(x)| \leq \lam \E^{-c r}.
 \ee
\end{lemma}

\begin{proof}
 This follows from \eqref{eq:asfr2} respecively \eqref{eq:asfr}.
\end{proof}

\begin{lemma}\label{lem:Homtiom}
 Let $r, s > 0$, and $\omega = \ti{\omega} \pmod{\Lambda_{r + s}(x) ^c}$.
 Then
 \be
  \|H_{\lam,\omega}^{\Lambda_r(x)} - H_{\lam,\ti{\omega}}^{\Lambda_r(x)}\|
   \leq \lam \E^{- c s}.
 \ee
\end{lemma}

\begin{proof}
 By assumption, we have for $y \in \Lambda_r(x)$
 that $\omega = \ti{\omega} \pmod{\Lambda_s(y)^c}$.
 The claim now follows by the previous lemma.
\end{proof}

In order to state the next lemma, we define given $r$ the number
$\ol{r}$ by
\be\label{eq:defrbar}
 \ol{r} = \left\lceil\left(1 + \frac{4 \gamma}{c}\right) r + \frac{\log(\lam)}{c}\right\rceil,
\ee
 
\begin{lemma}\label{lem:LamrstoLamr}
 Assume \eqref{eq:asrgamp} and that
 \be
  \Lambda_r(x)\text{ is $(\gamma,\tau,p)$-suitable for }H_{\lam, \omega} - E.
 \ee
 Furthermore assume that
 \be
  \omega = \ti{\omega} \pmod{\Lambda_{\ol{r}}(x)^c},
 \ee
 Then $\Lambda_r(x)$ is $(\gamma,\tau,p-1)$-suitable for $H_{\lam, \ti{\omega}} - E$.
\end{lemma}

\begin{proof}
 By Lemma~\ref{lem:Homtiom} with $s = \frac{4 \gamma}{c} \cdot r + \frac{\log(\lam)}{c}$
 and \eqref{eq:asrgamp}, we obtain
 $$
  \|H_{\lam,\omega}^{\Lambda_r(x)} - H_{\lam,\ti{\omega}}^{\Lambda_r(x)}\| \leq \E^{-4 \gamma r}.
 $$
 By Lemma~\ref{lem:suitstable}, the claim follows.
\end{proof}

We furthermore note

\begin{lemma}\label{lem:proprbar}
 We have
 \be
  \ol{r} \leq T_0 \max(r,\frac{\log(\lam)}{c}),\quad T_0 = \left(2  + \frac{4\gamma}{c}\right).
 \ee
\end{lemma}

\begin{proof}
 This follows from \eqref{eq:defrbar}.
\end{proof}

We will now begin with the probabilistic arguments.
For $n \in \Lambda_R(0)$ and $r \geq 1$, introduce $X^{r}_{n} = X^{r}_{n}(\gamma,\tau,p, E)$
as the set of $\omega$
such that for some $\ti{\omega} = \omega \pmod{\Lambda_{\ol{r}}(n)^{c}}$
we have
\be
 \Lambda_{r}(n)\text{ is not $(\gamma,\tau,p-1)$-suitable for }H_{\lam,\ti{\omega}} - E.
\ee

\begin{lemma}\label{lem:propXn}
 Assume \eqref{eq:asrgamp} and 
 \be\label{eq:probnotsuitableleqeps}
  \mathbb{P}(\Lambda_r(0)\text{ is not $(\gamma,\tau,p)$-suitable for } H_{\lam,\omega} - E) \leq \eps.
 \ee
 Then $\mathbb{P}(X^r_n) \leq \eps$. 

 Furthermore, for
 $|m - n|_{\infty} \geq 2 \ol{r} + 1$, $X_n^r$ and $X_m^r$
 are independent.
\end{lemma}

\begin{proof}
 By assumption and Lemma~\ref{lem:LamrstoLamr}, we have that
 $\mathbb{P}(X^r_n) \leq \eps$. The independence
 follows from $\Lambda_{\ol{r}}(n) \cap \Lambda_{\ol{r}}(m) = \emptyset$
\end{proof}

This implies

\begin{proposition}\label{prop:probatmostonebadcube}
 Assume \eqref{eq:asrgamp} and \eqref{eq:probnotsuitableleqeps}.
 Then for $R \geq r$ and $K \geq 1$ there exists a set $\mathcal{B}_{1}^{R,K}$ with the 
 following properties.
 \begin{enumerate}
  \item $\mathbb{P}(\mathcal{B}_{1}^{R,K}) \leq \frac{1}{(K+1)!} \left((3 R)^{d} \cdot \eps\right)^{K+1}$.
  \item For $\omega \notin \mathcal{B}_{1}^{R,K}$
   there exist $0\leq L \leq K$ and $m_{\omega}^{1}, \dots m_{\omega}^{L}$ such that
   for 
   \be
     \Lambda_r(n) \subseteq \Lambda_R(0) \setminus \bigcup_{k=1}^{L} \Lambda_{2 \ol{r}}(m_{\omega}^{k}),
   \ee
   we have
   \be
    \Lambda_r(n)\text{ is $(\gamma,\tau,p-1)$-suitable for }H_{\lam, \omega} - E.
   \ee
 \end{enumerate}
\end{proposition}

The conclusion of this proposition has to be understood as
{\em there exist at most $K$ bad cubes with large probability}.
Of course in order for this statement to be true, we need
that $\eps < (3R)^{d}$.

We now begin the proof of Proposition~\ref{prop:probatmostonebadcube}.
Let $m_k \in \Lambda_R(0)$ for $k = 1, \dots, K + 1$,
which satisfy for $k \neq \ell$ 
\be\label{eq:LamnLammdisjoint}
 \Lambda_{\ol{r}}(m_k)  \cap \Lambda_{\ol{r}}(m_\ell) = \emptyset.
\ee
Denote the collection $\ul{m} = \{m_k\}_{k=1}^{K + 1}$.
Introduce $\mathcal{B}_{r, K}^{\ul{m}}$ as the set of $\omega$ such that
for $k = 1, \dots, K+1$, we have
\be
 \Lambda_{r}(m_k)\text{ is not $(\gamma,\tau, p-1)$-suitable for }H_{\lam,\omega} - E.
\ee
We have

\begin{lemma}\label{lem:suitindep}
 Assume \eqref{eq:asrgamp}, and \eqref{eq:probnotsuitableleqeps}.
 Let $\ul{m} \in \Lambda_R(0)^{K+1}$ satisfy \eqref{eq:LamnLammdisjoint}. Then
 \be
  \mathbb{P}(\mathcal{B}_{r, K}^{\ul{m}}) \leq \eps^{K+1}.
 \ee
\end{lemma}

\begin{proof}
 One shows that
 $$
  \mathcal{B}_{r,K}^{\ul{m}} \subseteq \bigcap_{k =1}^{K+1} X^{m_k}.
 $$
 The claim now follows by Lemma~\ref{lem:propXn}.
\end{proof}

\begin{proof}[Proof of Proposition~\ref{prop:probatmostonebadcube}]
 Introduce $\mathcal{B}_1^{R,K}$ as the union over all possible choices
 of $\mathcal{B}_{r,K}^{\ul{m}}$. Since the number of these choices
 is bounded by
 $$
  \frac{1}{(K+1)!} (\# \Lambda_R(0))^{K+1} \leq \frac{1}{(K+1)!} (3R)^{(K+1)\cdot d},
 $$
 we obtain that (i) holds.
 
 Let us now check (ii). Take $\omega \notin \mathcal{B}_1^{R,K}$.
 Denote by $\{m_{\omega}^{j}\}_{j=1}^{J}$ a maximal collection of $m \in \Lambda_R(0)$
 such that \eqref{eq:LamnLammdisjoint} holds and
 $$
  \Lambda_r(m_{\omega}^{j})\text{ is not $(\gamma,\tau,p-1)$-suitable for }
   H_{\lam,\omega} - E.
 $$
 If $J \leq K$, we are done. Otherwise, we have a contradiction to
 $\omega \notin \mathcal{B}_{r,K}^{\{m_{\omega}^{j}\}_{j=1}^{K+1}}$,
 which finishes the proof.
\end{proof}

%
%
%

\section{Obtaining exponential decay of the off-diagonal terms}
\label{sec:green}

In this section, we demonstrate how to obtain estimates on the off-diagonal
elements, using an iteration of the resolvent equation. The results of this
section are again independent of the specific form of $H_{\lam,\omega}$,
so we use $H$ to denote some Schr\"odinger operator.
Introduce for $x,y \in \Lambda_r(n)$
the {\em Green's function}
\be
 G_{E}^{\Lambda_{r}(n)}(x,y) = \spr{e_x}{(H^{\Lambda_{r}(n)} - E)^{-1} e_y}.
\ee
We furthermore introduce the {\em boundary} of $\Lambda_r(n)$ in $\Lambda_R(n)$
by
\be
 \partial^{\Lambda_R(0)}\Lambda_r(n) = \{(x,y):\quad x\in\Lambda_r(n),\ y \in \Lambda_R(0) \setminus \Lambda_r(n),\ |x-y|_1 = 1\}.
\ee
For $x \in \Lambda_r(n)$ and $y \in \Lambda_R(0) \setminus \Lambda_r(n)$,
we obtain from the second resolvent equation
\be\label{eq:geomresolv}
 G_{E}^{\Lambda_{R}(0)}(x,y) = - \sum_{(u, v)\in\partial^{\Lambda_R(0)}\Lambda_r(n)}
  G_{E}^{\Lambda_{r}(n)}(x,u) \cdot G_{E}^{\Lambda_{R}(0)}(y,v) .
\ee
We will refer to this equation as the {\em geometric resolvent equation}.
See Section~5.3 in \cite{kirsch}.
We note the following consequence, which is more convenient for our applications
\begin{align}\label{eq:geomresolv2}
 |G_{E}^{\Lambda_{R}(0)}(x,y)| \leq &\left(\max_{n \in \Lambda_{r+1}(n)  \cap \Lambda_R(0)} |G_{E}^{\Lambda_R(0)}(y,n)| \right) \\
 \nn  &\cdot \left( \#(\partial_{-}^{\Lambda_R(0)} \Lambda_r(0)) \max_{n \in \partial_{-}^{\Lambda_R(0)} \Lambda_r(n)} |G_{E}^{\Lambda_{r}(n)}(x,n)| \right) .
\end{align}
Here $\partial_{-}^{\Lambda_R(0)} \Lambda_r(n)$ denotes
\be
 \partial_{-}^{\Lambda_R(0)} \Lambda_r(n) = \{x \in \Xi:\quad \exists y \in \Lambda_R(0) \setminus \Lambda_r(n):\quad |x -y|_1 = 1\}.
\ee
We note that $\partial_{-} \Lambda_r(n) = \partial_{-}^{\Z^d} \Lambda_r(n)$.
In order to illustrate the use of \eqref{eq:geomresolv2}, we first
prove the following, which is similar to Proposition~10.4. in \cite{kirsch}.

\begin{proposition}\label{prop:resolvitersimple}
 Let $R \geq r \geq 1$, $\tau\in (0,1)$, $\gamma>0$, and $p \geq 0$. Assume for
 $\Lambda_r(n) \subseteq \Lambda_R(0)$ that
 \be\label{eq:condLamrissuit}
  \Lambda_r(n)\text{ is $(\gamma,\tau,0)$-suitable for } H - E
 \ee
 and 
 \be
  \|(H^{\Lambda_R(0)} - E)\|\leq \frac{1}{2^p} \E^{R^{\tau}}.
 \ee 
 Then
 \be
  \Lambda_R(0)\text{ is $(\hat{\gamma},\tau,p)$-suitable for } H - E
 \ee
 where
 \be
  \hat{\gamma} = \gamma 
  \left(1 - \frac{1}{r + 1} -  \frac{10}{R}\left(\gamma r + R^{\tau} + d \log(3) \right)\right).
 \ee
\end{proposition}

\begin{proof}
 We have to check that for $x, y \in \Lambda_R(0)$ with $|x - y| \geq \frac{R}{10}$,
 we have exponential decay of the off-diagonal elements. By \eqref{eq:geomresolv2}
 and \eqref{eq:condLamrissuit}, we obtain for any $u \in \Lambda_r(n)$ 
 and $v \in \Lambda_R(0) \setminus \Lambda_r(n)$ that
 $$
  |G_{E}^{\Lambda_{R}(0)}(u,v)| \leq \E^{-\gamma r} \left(\max_{n \in \Lambda_{r+1}(u) \cap \Lambda_R(0)} |G_{E}^{\Lambda_R(0)}(v,n)| \right).
 $$
 Suppose for $k \geq 1$ that $v \notin \Lambda_R(0) \setminus \Lambda_{k (r+1) r}(u)$,
 then by iterating the previous equation we find
 $$
  |G_{E}^{\Lambda_{R}(0)}(u,v)| \leq \E^{-\gamma r k} \left(\max_{n \in \Lambda_{k (r+1)}(u) \cap \Lambda_R(0)} |G_{E}^{\Lambda_R(0)}(v,n)| \right).
 $$
 We now apply this formula to the case when $u = x$ and $v = y$.
 Then, for
 $$
  k = \left\lfloor \frac{|x - y|}{r + 1} \right\rfloor \geq \frac{|x - y|}{r + 1} - 1
 $$
 we obtain that $y \notin \Lambda_R(0) \setminus \Lambda_{k (r+1) r}(y)$.
 Thus using that $|G_{E}^{\Lambda_R(0)}(n,m)| \leq \frac{1}{2^{p}} \E^{R^{\tau}}$
 for any $n,m \in \Lambda_R(0)$ by assumption, that
 $$
  |G_{E}^{\Lambda_{R}(0)}(x,y)| \leq \frac{1}{2^{p}}
   \E^{- \frac{r}{r + 1} \gamma |x -y| + \gamma r + R^{\tau}}.
 $$
 The claim follows.
\end{proof}

A close inspection of the argument in this proof shows that it is
slightly wrong. In order to make it completely rigorous, one would
need to replace $\Lambda_r(n)$ in \eqref{eq:geomresolv2} by
the shifted cube $\Lambda^{R}_{r}(n)$ defined in Definition~\ref{def:shiftcube}.
That this does not create problems follows from
\be
 \partial^{\Lambda_R(0)} \Lambda^{R}_{r}(n) \cap \partial_-\Lambda_R(0) = \emptyset.
\ee

Although Proposition~\ref{prop:resolvitersimple} illustrates well,
what we will do it is not good enough yet. In view of 
Proposition~\ref{prop:probatmostonebadcube}, we will need
to allow for $K$ regions $\Lambda_{s_k}(m_k) \subseteq \Lambda_R(0)$ such
that we only know that every
$$
 \Lambda_r(n) \subseteq 
  \Lambda_R(0) \setminus \left(\bigcup_{k=1}^{K} \Lambda_{s_k}(m_k)\right)
$$
is $(\gamma,\tau,0)$-suitable. We will now give such a result,
which is similar to Theorem~10.20 in \cite{kirsch}.

\begin{theorem}\label{thm:resolv2}
 Let $a \geq 1$ and
 \be\label{eq:ineqskartk}
   s_k + a r \leq t_k.
 \ee
 Assume the following conditions.
 \begin{enumerate}
  \item $\Lambda_{t_k}(m_k) \subseteq \Lambda_R(0)$,
  \item For $k \neq \ell$
   \be
    \Lambda_{t_k + a r}(m_k) \cap \Lambda_{t_{\ell} + a r}(m_{\ell}) = \emptyset.
   \ee
  \item For every subcube 
   \be
    \Lambda_r(n) \subseteq \Lambda_R(0) \setminus \left(\bigcup_{k=1}^{K} \Lambda_{s_k}(m_k)\right)
   \ee
   that
   \be
    \Lambda_r(n)\text{ is $(\gamma,\tau,0)$-suitable for } H - E
   \ee
  \item For $1 \leq k \leq K$
   \be\label{eq:resolvbddagamr}
    \|(H^{\Lambda_{t_k}(m_k)} - E)^{-1} \| \leq \E^{a \gamma r}
   \ee
  \item On the resolvent of the whole cube
   \be
    \|(H^{\Lambda_R(0)}-E)^{-1}\| \leq \frac{1}{2^{p}} \E^{R^{\tau}}.
   \ee
 \end{enumerate}
 Then
 \be
  \Lambda_R(0)\text{ is $(\hat{\gamma},\tau,p)$-suitable for } H - E
 \ee
 where
 \be
  \hat{\gamma} = \gamma \left(1 - \frac{1}{r + 1} - \frac{10}{R}\left(3\sum_{k=1}^{K} t_k + \gamma r + R^{\tau} + d \log(3)\right)\right).
 \ee
\end{theorem}

In order to prove this theorem, we will introduce some additional
notation. Given $t \geq s \geq 1$, we introduce the annulus
\be
 \mathcal{A}_{s,t}^{R}(n) = (\Lambda_t(n) \setminus \Lambda_s(n))\cap\Lambda_R(0)
  = \{x\in\Lambda_R(0):\quad s+1\leq |x-n|_{\infty} \leq t\}.
\ee
We also introduce the abbreviation
\be
 \mathcal{A}_t^{R}(n) = \mathcal{A}_{t-1,t}^{R}(n)
  = \{x\in\Lambda_R(0):\quad |x-n|_{\infty} = t\}.
\ee
We note $\mathcal{A}^{R}_{t}(n) = \partial_{-}^{\Lambda_R(0)}\Lambda_t(n)$.

\begin{lemma}\label{lem:removebadsets}
 Let $x \in \Lambda_{t_k}(m_k)$ and $y \in \Lambda_R(0) \setminus \Lambda_{t_k + a r}(m_k)$.
 Assume \eqref{eq:resolvbddagamr}, then
 \be
  |G_{E}^{\Lambda_R(0)}(x,y)| \leq \max_{n \in \mathcal{A}_{t_k - ar, t_k + ar}^{R}(m_k)} |G_{E}^{\Lambda_R(0)}(y,n)|.
 \ee
\end{lemma}

\begin{proof}
 Using \eqref{eq:geomresolv} and then \eqref{eq:geomresolv2}, we compute
 \begin{align*}
  |G_{E}^{\Lambda_R(0)}(x,y)| & \leq \|(H^{\Lambda_{t_k}(m_k)} - E)^{-1}\| 
  \cdot \sum_{v\in \partial_{-}^{\Lambda_R(0)} \Lambda_{t_k}(m_k)} |G_{E}^{\Lambda_R(0)}(y,v)| \\
  & \leq \|(H^{\Lambda_{t_k}(m_k)} - E)^{-1}\| \cdot \E^{-\gamma r} 
  \max_{n \in \mathcal{A}_{t - r, t + r}^{R}(n_0)} |G_{E}^{\Lambda_R(0)}(y,n)|,
 \end{align*}
 where we used that for $n$ with $|n - n_0|_{\infty} = t_k + 1$
 we have
 $$
  \Lambda_r(n)\text{ is $(\gamma,\tau,0)$-suitable for } H - E.
 $$
 The claim follows by iterating the above procedure $a$ times.
\end{proof}

Given $x,y \in \Lambda_R(0)$ satisfying $|x - y| \geq \frac{R}{10}$,
we introduce $L$ maximal such that $y \notin \Lambda_{L r}(x)$
or alternatively
\be
 y \in \mathcal{A}_{L r, (L+1)r}^{R}(x).
\ee
We then call $1 \leq \ell \leq L$ {\em bad}, if 
\be
 \mathcal{A}_{\ell r}^{R}(x) \cap \Lambda_{t_k + a r}(m_k) \neq \emptyset
\ee
for some $k$. We will need

\begin{lemma}
 We have
 \be
  \frac{\#(\text{bad } \ell)}{|x - y|} \leq \frac{30}{R} \left(\sum_{k=1}^{K} \left(\frac{t_k}{r} + a \right) \right).
 \ee
\end{lemma}

\begin{proof}
 For $1 \leq k \leq K$, there are at most $3\left(\frac{t_k}{r}+a\right)$
 choices of $\ell$ such that
 $$
  \mathcal{A}_{\ell r}^{R}(x) \cap \Lambda_{t_k + a r}(m_k) \neq \emptyset.
 $$
 Summing over $k$ and using $|x - y| \geq \frac{R}{10}$ implies
 the claim.
\end{proof}

\begin{proof}[Proof of Theorem~\ref{thm:resolv2}]
 We now proceed as in the proof of Proposition~\ref{prop:resolvitersimple},
 except we use Lemma~\ref{lem:removebadsets} to deal with the bad $\ell$.
\end{proof}

%
%
%

\section{A combinatorial result}
\label{sec:combinatoric1}

In this section, we will derive a combinatorial results about
boxes in $\Z^d$, which will be used to ensure the geometric
conditions (i) and (ii) in Theorem~\ref{thm:resolv2}. We begin
by introducing some notation.

Given $R \geq 1$, $n \in \Lambda_R(0)$, and $0 \leq t \leq R$,
we denote by $n^t_R\in\Lambda_R(0)$ the point
\be
 (n^t_R)_j = \begin{cases} - R + t, & n_j < -R + t; \\
  R - t, & n_j > R - t; \\
  n_j, &\text{otherwise}.\end{cases}
\ee
This choice is made such that $\Lambda_t(n^t_R) \subseteq \Lambda_R(0)$
and 
$$
 \#(\Lambda_t(n^t_R) \cap \Lambda_t(n))
$$
is maximal under all subcubes $\Lambda_t(m) \subseteq \Lambda_R(0)$.

\begin{definition}\label{def:shiftcube}
 Given $0 \leq t \leq R$, $s\geq 0$, and $n \in \Lambda_t(n)$,
 we introduce the shifted cube
 \be
  \Lambda^{R,t}_s(n) = \Lambda_s(n^{t}_{R}).
 \ee
 Furthermore, we define $\Lambda^{R}_t(n) = \Lambda_{t}^{R,t}(n)$.
\end{definition}

We always have $\Lambda_t^{R}(n)\subseteq\Lambda_R(0)$, but this
must nopt be true for $\Lambda_s^{R,t}(n)$ if $s > t$.
In order to pass from the probabilistic result to what we need
in the analytic part, we will need the following proposition.

\begin{proposition}\label{prop:combinatoric}
 Let $Q \geq K + 1$ and $R \geq r \geq 1$. 
 Given $K$ points $n_1, \dots, n_K \in \Lambda_R(0)$ and 
 $Q$ length scales $r \leq r_1 \leq s_1 \leq \dots \leq r_Q \leq s_Q \leq R$ satisfying
 \be\label{eq:condrq1rq}
  r_{q + 1} \geq 3 s_q.
 \ee
 There exists a sequence of points $m_1,\dots,m_J$ and 
 numbers $q_1,\dots,q_J$ with $J\leq K$ such that
 \begin{enumerate}
  \item For $i \neq j$
   \be
    \Lambda_{s_{q_i}}^R(m_i) \cap \Lambda_{s_{q_j}}^R(m_j) = \emptyset
   \ee 
  \item For each $1 \leq k \leq K$, there exists $j$ such
   that
   \be
    \Lambda_{r}(n_k) \subseteq \Lambda_{r_{q_j}}^R(m_j).
   \ee
 \end{enumerate}
\end{proposition}

We will need the following lemma.

\begin{lemma}
 Let $q \leq \ti{q}$ and $m, \ti{m} \in \Lambda_R(0)$.
 Assume \eqref{eq:condrq1rq} and
 \be
  \Lambda_{s_q}(m) \cap \Lambda_{s_{\ti{q}}}(\ti{m}) \neq 0.
 \ee
 Then 
 \be
  \Lambda_{r_q}(m) \subseteq \Lambda_{r_{\ti{q} + 1}}(\ti{m}).
 \ee
\end{lemma}

\begin{proof}
 Let $x \in \Lambda_{s_q}(m) \cap \Lambda_{s_{\ti{q}}}(\ti{m})$.
 For any $y \in \Lambda_{r_q}(m)$, we have $|y - x|_{\infty} \leq 2 r_q$.
 Thus
 $$
  |y - \ti{m}|_{\infty} \leq |y - x|_{\infty} + |y- \ti{m}|_{\infty}
   \leq s_{\ti{q}} + 2 r_q \leq 3 s_{\ti{q}} \leq r_{\ti{q} + 1}.
 $$
 This implies the claim.
\end{proof}

We now come to

\begin{proof}[Proof of Proposition~\ref{prop:combinatoric}]
The proof of this result is constructive, and proceeds by induction.

{\em Base case of the induction:} Assume $K = 1$. Then $Q = 1$,
$m_1 = n_1$, and $q_1 = 1$. It is clear that
\be
 \Lambda_r(n_1) \subseteq \Lambda_{r_{q_1}}(m_1).
\ee
This finishes the base case.

Let us now do the {\em induction step}. So assume we have a solution
for $K - 1$, and we are given the point $n_{K} \in \Lambda_R(0)$.
Then there are several cases, depending on the location of 
$\Lambda_r(n_K)$ with respect to the $\Lambda_{r_{q_j}}(m_j)$.

{\bf Case 1:} Assume there exists $j$ such that
\be
 \Lambda_{r}(n_K) \subseteq \Lambda_{r_{q_j}}(m_j).
\ee
Then we do nothing.

{\bf Case 2:} Assume that for $1 \leq j \leq J$, we have
\be
 \Lambda_{s_1}(n_K) \cap \Lambda_{s_{q_j}}(m_j) = \emptyset.
\ee
Then, we add $n_k$ and $1$ to the $m_j$ and $q_j$.

{\bf Case 3:} There exists $1 \leq i \leq J$ such that
\be
 \Lambda_{s_1}(n_K) \cap \Lambda_{s_{q_i}}(m_i) \neq \emptyset.
\ee
Then, we choose an $i$ satisfying the previous condition
such that $q_i$ is maximal. We then increase $q_i$ by $1$.
It is clear that now our second condition holds since we will
have $\Lambda_{r}(n_K) \subseteq \Lambda_{r_{q_i+1}}(m_i)$.
But it is not clear that the $\Lambda_{s_{q_j}}(m_j)$ are
still disjoint from $\Lambda_{s_{q_i+1}}(m_i)$.

 If for some $j \neq i$
 $$
  \Lambda_{s_{q_j}}(m_j) \cap \Lambda_{s_{q_i+1}}(m_i) \neq \emptyset,
 $$ 
 then we increase the larger $q_u$ by $1$, and remove the other
 cube. This is possible by the previous lemma.
 It is clear that this process terminates, and generates
 the goal of the inductive scheme.

 We have already seen the existence. Let us now discuss the
 bound on $Q$. Each time we increase one of the $q_j$, we remove one
 of the $n_k$ from the list of $m_j$. Since, we only have $K$
 many $n_k$, we see that we can increase $q_j$ at most $K$
 times. This finishes the proof.
\end{proof}

%
%
%

\section{Under the assumption of a Wegner estimate}
\label{sec:asswegner}

In this section, we will give the multi-scale argument under the assumption
of a Wegner estimate. We will choose $\tau = \frac{1}{2}$, $p = 3$,
\be
 K = 1,\quad R \in \left[r^{1 + \frac{1}{5 d}}, r^{1 + \frac{1}{2 d}}\right],
 \quad \alpha = d \frac{2d + 2}{2 d- 1}.
\ee
The  first step is the following probabilistic estimate

\begin{lemma}
 Assume that for $r \geq r_0$ and some fixed $C > 0$, $b \geq 0$,
 and $\sigma > 0$.
 \be
  \mathbb{E}\left(\tr\left(P_{[E - \eps, E + \eps]}\left(H_{\lam,\omega}^{\Lambda_r(0)}\right)\right)\right)
   \leq \frac{C \cdot \left(\#\Lambda_r(0)\right)^b}{(\log(\eps^{-1})^{2 \alpha + 2 d (b+1) + \sigma}}.
 \ee
 Assume $R \geq R_0(d,C,\sigma)$.
 Then there exists a set $\mathcal{B}_W^{R}$ with the properties
 \begin{enumerate}
  \item $\mathbb{P}(\mathcal{B}_W^{R}) \leq \frac{1}{2} \cdot \frac{1}{R^{\alpha}}$.
  \item For $\omega \notin \mathcal{B}_W^{R}$ and $\Lambda_s(n) \subseteq \Lambda_R(0)$, 
   we have
   \be
    \|(H_{\lam,\omega}^{\Lambda_s(n)} - E)^{-1}\| \leq \frac{1}{8}\E^{\sqrt{R}}.
   \ee
 \end{enumerate}
\end{lemma}

\begin{proof}
 The number of subcubes $\Lambda_s(n) \subseteq \Lambda_R(0)$ is bounded
 by
 $$
  R \cdot \#(\Lambda_R(0)) \leq 3^d \cdot R^{d+1}.
 $$
 By Markov's inequality, we have
 $$
  \mathbb{P}\left(\|(H_{\lam,\omega}^{\Lambda_s(n)} - E)^{-1}\| 
   \leq \frac{1}{8}\E^{\sqrt{R}} \right)
    \leq \mathbb{E}\left(\tr\left(P_{[E - \eps, E + \eps]}
     \left(H_{\lam,\omega}^{\Lambda_s(0)}\right)\right)\right),
 $$
 where $\eps = 8 \E^{-\sqrt{R}}$. The claim follows after some computations.
\end{proof}

Introduce
\be
 \mathcal{B}^{R} = \mathcal{B}_1^{R} \cup \mathcal{B}_W^{R},
\ee
where $\mathcal{B}_1^{R}$ denotes the set from Proposition~\ref{prop:probatmostonebadcube}.
If $r \geq 1$ is large enough, we have
\be
 \mathbb{P}(\mathcal{B}^{R}) \leq \frac{1}{R^{\alpha}}.
\ee

We introduce in view of Theorem~\ref{thm:resolv2}
\be
 a = \left\lceil \frac{\sqrt{R}}{\gamma r} \right\rceil.
\ee
We now introduce a sequence of scales
\be
 r_1 = 2 \ol{r},\quad s_{q} = r_{q}, \quad r_{q+1} = \max(s_{q} + a r, 3 s_q).
\ee
We now obtain

\begin{lemma}
 Let $\omega\notin\mathcal{B}^{R}$. There exist
 $\ti{m}_1,\dots, \ti{m}_J$ and $q_1, \dots, q_J$ such that
 \begin{enumerate}
  \item For 
   \be
    \Lambda_r(n) \subseteq \Lambda_R(0) \setminus \bigcup_{j=1}^{J} \Lambda_{r_{q_j}}(\ti{m}_j)
   \ee
   we have 
   \be
    \Lambda_r(n)\text{ is $(\gamma,\frac{1}{2},1)$-suitable for } H_{\lam,\omega} - E.
   \ee
  \item For $i \neq j$, we have
   \be
    \Lambda_{r_{q_i} + ar}(\ti{m}_i) \cap \Lambda_{r_{q_j} + ar}(\ti{m}_j) = \emptyset.
   \ee
 \end{enumerate}
\end{lemma}

\begin{proof}
 Denote by $m_{\omega}^{1}, \dots, m_{\omega}^{K}$ the sites from
 conclusion (ii) of Proposition~\ref{prop:probatmostonebadcube}.
 The claim then follows by applying Proposition~\ref{prop:combinatoric}
 to these and the length scales $r_q$ and $s_q$ defined above.
\end{proof}

\begin{proof}[Proof of Theorem~\ref{thm:indstepwegner}].
 Let $\omega \notin \mathcal{B}^{R}$. Then by Theorem~\ref{thm:resolv2}
 and some computations,
 we have
 $$
  \Lambda_R(0)\text{ is $(\hat{\gamma}, \frac{1}{2}, 2)$-suitable for } H_{\lam,\omega} - E,
 $$
 where
 $$
  \hat{\gamma} = \gamma \left(1 - \frac{200}{r^{\frac{1}{d}}} \right).
 $$
 This finishes the proof.
\end{proof}


%
%
%

%
%
%

\section{Estimating the norm of the resolvent}
\label{sec:resolv1}

In this section, we will discuss how to obtain estimates on the norm
of the resolvent. The key point is that, we will develop methods that
will allow us to exploit the knowledge obtained from Proposition~\ref{prop:probatmostonebadcube}
to find a replacement of the Wegner estimate. In particular, we will be
able to show that the norm of the resolvent of $H_{\lam,\omega} - E$
restricted to the union of $(\gamma,\tau,p)$-suitable cubes of size $r$ is
bounded by $\E^{4 r^{\tau}}$.
Since the results of this section are independent of the specific form
of the potential, we will again write $H$ instead of $H_{\lam,\omega}$.
 
We will need to introduce some notation. Given
two sets $\Lambda \subseteq \Xi \subseteq \Z^d$,
the {\em boundary} $\partial^{\Xi} \Lambda$ of $\Lambda$ in $\Xi$
is defined by
\be
 \partial^{\Xi} \Lambda = \{(x,y):\quad x \in \Lambda,\ y \in \Xi\setminus\Lambda,\ |x - y|_1 =1\}.
\ee
The relevance of the boundary comes from that if
$u$ solves $H^{\Xi} u = E u$, then we have for $n \in \Lambda$ that
\be\label{eq:unasgreen}
 u(n) = - \sum_{(x,y) \in \partial^{\Xi}\Lambda} \spr{e_n}{(H^{\Lambda} - E)^{-1} e_x} u(y).
\ee
We furthermore introduce the distance of $n \in \Z^d$ to
$\partial^{\Xi} \Lambda$ by
\be
 \dist(n, \partial^{\Xi}\Lambda) = \min_{(x,y) \in \partial^{\Xi}\Lambda} |n - x|_{\infty}.
\ee
We are now ready to introduce the following geometric
condition.

\begin{definition}\label{def:setacceptable}
 A set $\Xi \subseteq \Z^d$ is called $r$-acceptable, if 
 the following conditions hold.
 \begin{enumerate}
  \item The set $\Xi$ is finite.
  \item For any $x \in \Xi$ there exists a cube $\Lambda_r(n) \subseteq \Xi$ such
   that
   \be
    \dist(x, \partial^{\Xi} \Lambda_r(n)) \geq \frac{r}{10}.
   \ee
 \end{enumerate}
\end{definition}

If a set $\Xi$ is $r$-acceptable, we can apply \eqref{eq:unasgreen} to every $n \in \Xi$
such that the off-diagonal decay condition in the definition of suitability is meaningful.
Furthermore an $r$-acceptable set is always the union of cubes of size $r$.
We are now ready for

\begin{theorem}\label{thm:resolv1}
 Let $\Xi$ be $r$-acceptable. Assume for any cube $\Lambda_r(n) \subseteq \Xi$ that
 \be
  \Lambda_r(n)\text{ is $(\gamma,\tau,0)$-suitable for }H - E,
 \ee
 and the inequalities
 \be\label{eq:ineqresolv1}
  4 \#(\partial_{-} \Lambda_r(n)) \leq \E^{(r)^{\tau}},\quad 40 \leq \gamma r, \quad r^{\tau} \geq \log(2).
 \ee
 Then
 \be
  \|(H^{\Xi} - E)^{-1} \| \leq \E^{3 (r)^{\tau}}.
 \ee
\end{theorem}

By the results of Section~\ref{sec:green}, we obtain

\begin{corollary}
 Let $R \geq r \geq 1$, $\tau\in(0,1)$, and $p \geq 0$.
 Assume \eqref{eq:ineqresolv1}, $3r^{\tau} + p \log(2) \leq R^{\tau}$,
 and for every subcube $\Lambda_r(n) \subseteq \Lambda_R(0)$ that
 \be
  \Lambda_r(n)\text{ is $(\gamma,\tau,0)$-suitable for } H - E
 \ee
 Then
 \be
  \Lambda_R(0)\text{ is $(\hat{\gamma},\tau,p)$-suitable for } H - E
 \ee
 where $\hat{\gamma} = \gamma \cdot \left(1  - \frac{1}{r + 1} - \frac{10}{R}\left(\gamma r + R^{\tau} + d \log(3) \right) \right)$.
\end{corollary}

\begin{proof}
 Use Proposition~\ref{prop:resolvitersimple}.
\end{proof}

We furthermore have the following variant, which loosens the requirement
on all subcubes of size $r$ to be suitable. However, we will need to at
least require control on the norm of the resolvent on the exceptional
regions, which can be larger then size $r$.

\begin{theorem}\label{thm:resolv1b}
 Given $R \geq 1$, $\gamma > 0$, $\tau \in (0,1)$, $m_1, \dots, m_Q \in \Lambda_R(0)$,
 and $t_q\geq s_q+r$ for $1 \leq q \leq Q$. Assume the following
 \begin{enumerate}
  \item For $1\leq q \leq Q$, we have $\Lambda_{t_q}(m_q) \subseteq \Lambda_R(0)$.
  \item The set
   \be
    \Lambda_R(0) \setminus \left(\bigcup_{q=1}^{Q} \Lambda_{s_q+r}(m_q) \right)
   \ee
   is $r$-acceptable.
  \item For $q \neq \ti{q}$, we have
   \be
    \Lambda_{t_q + 2r+1}(m_q) \cap \Lambda_{t_{\ti{q}} + 2r+1}(m_{\ti{q}}) = \emptyset.
   \ee
  \item For any cube   
   \be 
    \Lambda_r(n) \subseteq \Lambda_R(0) \setminus \left(\bigcup_{q=1}^{Q} \Lambda_{s_q}(m_q) \right),
   \ee
   we have
   \be
    \Lambda_r(n)\text{ is $(\gamma,\tau,0)$-suitable for }H - E.
   \ee
  \item For $1 \leq q \leq Q$, we have
   \be\label{eq:asthmresolv1b}
    \|(H^{\Lambda_{t_q}(m_q)} - E)^{-1} \| \leq \E^{(t_q)^{\tau}}.
   \ee
  \item The inequalities
   \be\label{eq:ineqresolv1b}
    4 \#(\Lambda_R(0)) \leq \E^{(t_{\infty})^{\tau}},\quad
     \gamma r \geq 10 (t_{\infty})^{\tau} + 10 \log(2 \#(\partial_-(\Lambda_r(0))),
   \ee
   where $t_{\infty} = \max_{1 \leq q \leq Q} t_q$.
  \end{enumerate}
 Then
 \be
  \|(H^{\Lambda_R(0)} - E)^{-1}\| \leq \E^{5 (t_{\infty})^{\tau}}.
 \ee
\end{theorem}

It should be noted that (i), (ii), and (iii) restrict the
geometry of the sets $\Lambda_R(0)$, $\Lambda_{s_q}(m_q)$
and $\Lambda_{t_q}(m_q)$. Whereas (iv) and (v) consist of
the assumptions on the operator $H - E$.

\begin{remark}\label{rem:normalop}
 Theorem~\ref{thm:resolv1} and \ref{thm:resolv1b} remain
 valid if the potential of the Schr\"odinger operator is
 complex valued, since the proof only uses that $H$
 is normal.
\end{remark}

We now start with the proof of Theorem~\ref{thm:resolv1} and
Theorem~\ref{thm:resolv1b} for which one needs to specialize
$\Xi = \Lambda_R(0)$. We first
recall that
\be
 \|(H^{\Xi} - E)^{-1} \| = \frac{1}{\dist(E, \sigma(H^{\Xi}))}.
\ee
Hence, there exists $\hat{E}\in\sigma(H^{\Xi})$ such that
\be\label{eq:normHhatE}
 \|(H^{\Xi} - E)^{-1} \| = \frac{1}{|E - \hat{E}|}
\ee
and a solution $u$ of $H^{\Xi} u = \hat{E} u$ with
$\| u \|_{\ell^2(\Xi)} = 1$. The strategy of the proof will
to show that for $|E - \hat{E}|$ small no such solution can exist.

\begin{lemma}\label{lem:perturbElarge}
 Assume $\Lambda_r(n)$ is $(\gamma,\tau,0)$-suitable for $H - E$
 and \eqref{eq:ineqresolv1}. Let 
 \be
  |\hat{E} - E| \leq \E^{- 3 (r)^{\tau}}.
 \ee
 Then for $x,y \in \Lambda_r(n)$ satisfying $|x - y| \geq \frac{r}{10}$
 \be
  |\spr{e_x}{(H^{\Lambda_r(n)} - \hat{E})^{-1} e_y}| \leq \frac{1}{2 \#(\partial_-\Lambda_r(0))}.
 \ee
\end{lemma}

\begin{proof}
 By a computation as done in  the proof of Lemma~\ref{lem:suitstable},
 we find
 \begin{align*}
   |\spr{e_x}{(H^{\Lambda_r(n)} - \hat{E})^{-1} e_y}| = 
   &|\spr{e_x}{(H^{\Lambda_r(n)} - E)^{-1} e_y}|\\
   & + 2 |E - \hat{E}| \|(H^{\Lambda_r(n)} - E)^{-1}\|^2.
 \end{align*}
 By \eqref{eq:ineqresolv1} and the condition on $\hat{E}$,
 we find
 $$
  |E - \hat{E}| \|(H^{\Lambda_r(n)} - E)^{-1}\|^2 \leq \E^{- (r)^{\tau}} 
  \leq \frac{1}{4} \frac{1}{\#(\partial_-(\Lambda_r(0)))}.
 $$
 The claim follows.
\end{proof}

\begin{proof}[Proof of Theorem~\ref{thm:resolv1}]
 Let $\hat{E}$ be such that
 $$
  |\hat{E} - E| \leq \E^{- 3 (r)^{\tau}}.
 $$
 Let $u: \Xi\to\C$ solve $H^{\Xi} u = \hat{E} u$ with possibly $u = 0$.
 Since $\Xi$ is finite, we can choose $x$ such that
 $$
  |u(x)| = \max_{y\in\Xi} |u(y)|.
 $$
 Let $\Lambda_r(n)$ be the set from the $r$-acceptable property of $\Xi$
 for $x$. By \eqref{eq:unasgreen} and the previous lemma,
 we can conclude
 $$
  |u(x)| \leq \frac{1}{2} \sup_{(\ti{x}, \ti{y}) \in \partial^{\Xi}\Lambda_r(n)} |u(\ti{y})|
  \leq \frac{1}{2} |u(x)|.
 $$
 This is only possible if $u = 0$. Hence, we see that we must have
 $\dist(E, \sigma(H^{\Xi})) \geq \E^{-3 (r)^{\tau}}$.
 By \eqref{eq:normHhatE}, the claim follows.
\end{proof}

Similarly to Lemma~\ref{lem:perturbElarge}, one can show

\begin{lemma}\label{lem:perturbElarge2}
 Assume $\Lambda_r(n)$ is $(\gamma,\tau,0)$-suitable for $H - E$
 and \eqref{eq:ineqresolv1b}.
 Let 
 \be
  |\hat{E} - E| \leq \E^{- 5 (t)^{\tau}}.
 \ee
 Then for $x,y \in \Lambda_r(n)$ satisfying $|x - y| \geq \frac{r}{10}$
 \be
 |\spr{e_x}{(H^{\Lambda_r(n)} - \hat{E})^{-1} e_y}| 
   \leq \frac{1}{\#(\partial_- \Lambda_r(n))} \E^{-2 (t_{\infty})^{\tau}}.
 \ee
\end{lemma}

Let $\hat{E}$ satisfy
\be
 |E - \hat{E}| \leq \E^{-5 (t_{\infty})^{\tau}}.
\ee
Assume $u$ solves $H^{\Lambda_R(0)} u = \hat{E} u$ and $\|u\|_{\ell^2(\Lambda_R(0))} = 1$.
We can thus find $x_0 \in \Lambda_R(0)$ such that
\be
 |u(x_0)| = \max_{x\in\Lambda_R(0)} |u(x)|.
\ee
Since $\|u\|_{\ell^2(\Lambda_R(0))} = 1$, we have
\be
 |u(x_0)|^2 \geq \frac{1}{\#(\Lambda_R(0))}.
\ee
Our first goal will to localize $x_0$.

\begin{lemma}
 Assume (ii), (iv), and (vi). Let 
 \be
  x\in\Lambda_R(0)\setminus\left(\bigcup_{q=1}^{Q}\Lambda_{s_q+r}(m_q)\right).
 \ee
 Then, we have
 \be
  |u(x)| \leq \E^{-2 (t_{\infty})^{\tau}}.
 \ee
\end{lemma}

\begin{proof}
 By (ii) and (iv), we can apply Lemma~\ref{lem:perturbElarge2} to
 $\Lambda_r(x)$. The claim follows from \eqref{eq:unasgreen}.
\end{proof}

\begin{proof}[Proof of Theorem~\ref{thm:resolv1b}]
 By the previous lemma, there is $1 \leq q \leq Q$ such that
 $$
  x_0 \in \Lambda_{s_q + r}(m_q) \subseteq \Lambda_{t_q}(m_q).
 $$
 Define $v$ by
 $$
  v(x) = \begin{cases} u(x), &x \in \Lambda_{s_q+r}(m_q);\\ 0, &\text{otherwise}.\end{cases}
 $$
 We can now compute
 $$
  \|(H^{\Lambda_{t_q}(m_q)} - E) v\| \leq 
   \sqrt{\#(\partial_-\Lambda_{t_q}(m_q))} \E^{-2 (t_{\infty})^{\tau}}
    \leq \frac{1}{2 \sqrt{\#(\Lambda_R(0))}} 
     \E^{-(t_{\infty})^{\tau}} \leq 
      \frac{1}{2} \|v\| \E^{-(t_{\infty})^{\tau}}.
 $$
 By \eqref{eq:normHhatE}, this implies 
 $\|(H^{\Lambda_{t_q}(m_q)} - E)^{-1}\| \geq 2 \E^{(t_{\infty})^{\tau}}$.
 This contradicts \eqref{eq:asthmresolv1b} showing no such $u$ can exist.
 The claim follows.
\end{proof}

%
%
%

\section{Cartan's lemma}
\label{sec:cartan}

In this section, we derive a simple matrix valued version of
Cartan's lemma. Cartan's lemma was first used in the spectral
theory context by Goldstein and Schlag \cite{gs01}
and in the form we use it by Bourgain, Goldstein, and Schlag in
\cite{bgs2}. Then further improved on by Bourgain in 
\cite{b2002}, \cite{bbook}, \cite{b2007}, and \cite{b2009}.
In particular \cite{b2009} is important for us, since
it contains a version of Cartan's lemma where the constant
depends nicely on the dimension. We introduce for $r > 0$
\be
 \mathbb{D}_r = \{z\in\mathbb{C}:\quad |z|<r\}
\ee
and we recall $\mathbb{D} = \mathbb{D}_6$.

\begin{proposition}\label{prop:cartanbourgain}
 Let $f: \mathbb{D}_{2\E}^n \to \C$ be an analytic function
 satisfying
 \be
  \|f\|_{L^\infty(\mathbb{D}_{2\E}^n)} \leq 1,
   \quad |f(0)| \geq \eps
 \ee
 Then for $s>0$
 \be\label{eq:cartesti}
  |\{x \in [-1, 1]^n:\quad |f(x)| \leq \E^{-s}\}|
   \leq 60 \E^3 n^{3/2} 2^n \exp\left(-\frac{s}{\log(\eps^{-1})} \right).
 \ee
\end{proposition}

This proposition is a variant of Lemma~1 in \cite{b2009}. We note
that the $2^n$ on the right hand side corresponds to the measure
of $[-1,1]^n$. So that the constant has a slight dependence
on the dimension. It would be interesting to determine
the optimal value. We now turn to the proof of
Proposition~\ref{prop:cartanbourgain}. For this,
we first recall the one dimensional version of Cartan's lemma.

\begin{theorem}\label{thm:cartanlemma}
 Let $f: \mathbb{D}_{2\E} \to \C$ be analytic satisfying
 \be
  \sup_{|z|\leq 2\E} |f(z)| \leq 1,\quad |f(0)| \geq \eps.
 \ee
 Then for $s>0$
 \be
  |\{x\in \left[-1, 1 \right]:\quad |f(x)| \leq \E^{-s}\}|
   \leq 30 \E^{3} \exp\left(- \frac{s}{\log(\eps^{-1})}\right).
 \ee
\end{theorem}

\begin{proof}
 We apply Theorem~11.3.4. in Levin's book \cite{lev} to $g(z) = \frac{1}{\eps} f(z)$
 with $R = 1$, $M_g(2\E) \leq \frac{1}{\eps}$ and
 $$
  \log(\eta) = - \frac{s}{\log(\eps^{-1})} + \log(15 \E^3).
 $$
 The claim follows using $|f(x)| \leq |g(x)|$.
\end{proof}

We now come to

\begin{proof}[Proof of Proposition~\ref{prop:cartanbourgain}]
 We can write $x \in \R^n$ as $x = r \vartheta$, where
 $\vartheta \in S^{n-1}$ and $r > 0$. Define $r_{\vartheta}^{\mathrm{max}}$
 as the maximal choice of $r > 0$ such that $r \vartheta\in [-1,1]^n$.
 Then one finds that
 $$
  2^n = \left|\left[-1, 1\right]^n \right|
  = c_n \int_{S^{n-1}} \int_0^{r_{\vartheta}^{\mathrm{max}}} r^{n-1} d r d\vartheta
  = c_n \int_{S^{n-1}} \frac{1}{n} (r_{\vartheta}^{\mathrm{max}})^n d \vartheta.
 $$
 In particular that $c_n \int_{S^{n-1}} r(\vartheta)^n d\vartheta = n 2^n$.
 Introduce
 $$
  E_s = \{x \in [-1,1]^n:\quad |f(x)| \leq \E^{-s}\}.
 $$ 
 We have
 $$
  |E_s| = c_n \int_{S^{n-1}} \int_{0}^{r_{\vartheta}^{\mathrm{max}}} 
    \chi_{E_s}(r \vartheta) r^{n-1} d r d\vartheta.
 $$
 We will now analyze the inner integral for fixed $\vartheta\in S^{n-1}$.
 Define $g_{\vartheta}(z) = f(\vartheta r_{\vartheta}^{\mathrm{max}} z)$, then we have
 that
 \begin{align*}
  \int_{0}^{r_{\vartheta}^{\mathrm{max}}} \chi_{E_s}(r \vartheta) r^{n-1} d r
  & = (r_\vartheta^{\mathrm{max}}) \cdot \int_{0}^{1} \chi_{F_{s}(\vartheta)}(x) 
    \left(x r_{\vartheta}^{\mathrm{max}} \right)^{n-1} d x \\
  & \leq (r_{\vartheta}^{\mathrm{max}})^n \cdot \int_{0}^{1} \chi_{F_{s}(\vartheta)}(x) dx \\
  & = (r_{\vartheta}^{\mathrm{max}})^n |F_{s}(\vartheta)|
 \end{align*}
 where
 $$
  F_{s}(\vartheta) = \{x \in [-1,1]:\quad |g_{\vartheta}(x)| \leq \E^{-s}\}.
 $$
 Observe now that $r_{\vartheta}^{\mathrm{max}} \vartheta \D_{2\E} \subseteq \D_{2\E}^n$,
 $|g_{\vartheta}(0)| \geq \eps$, and $\sup_{|z| \leq 2 e} |g_{\vartheta}(z)| \leq 1$.
 Hence, we obtain by Cartan's lemma that
 $$
  |F_s| \leq 30 \E^{3} \exp\left(-\frac{s}{\log(\eps^{-1})}\right).
 $$
 The claim now follows using that $r_{\vartheta}^{\mathrm{max}} \leq \sqrt{2 n}$.
\end{proof}

We will now derive a matrix-valued version of this result. The main
idea is to apply Proposition~\ref{prop:cartanbourgain} to the
determinant of the matrix.

\begin{theorem}\label{thm:matcartan}
 Let $B \geq N$, $D\geq 1$, $s>0$, and
 $A: \mathbb{D}^n \to \C^{N \times N}$ be an analytic function.
 Assume there exists $x_0 \in [-\frac{1}{2}, \frac{1}{2}]^n$
 such that
 \be
  \sup_{z \in \mathbb{D}^n} \|A(z) \|\leq B,\quad\text{and}\quad \|A(x_0)^{-1}\| \leq D .
 \ee
 Assume 
 \be
  \frac{s}{N} \geq 20 n \log(B \cdot D).
 \ee
 Then
 \be
  |\{x \in [-\frac{1}{2}, \frac{1}{2}]^n:\quad \|A(x)\| \geq \E^{s} \}|
    \leq \exp\left(-\frac{1}{2} \frac{s}{N \log(B D)} \right).
 \ee
\end{theorem}

In order to pass from the determinant to information on the matrix,
we use the following lemma.

\begin{lemma}\label{lem:detestis}
 Let $A$ be an $N \times N$ matrix.
 \begin{enumerate}
  \item $|\det(A)| \leq \|A\|^{N}$.
  \item $|\det(A)| \geq \frac{1}{\|A^{-1}\|^{N}}$.
  \item $\|A^{-1}\| \leq \frac{N \|A\|^{N-1}}{|\det(A)|}$.
 \end{enumerate}
\end{lemma}

\begin{proof}
 Denote by $\lam_j$ the eigenvalues of $A$ and order them
 by modulus
 $$
  0 \leq |\lam_1| \leq |\lam_2| \leq \dots \leq |\lam_N|.
 $$
 By $\det(A) = \prod_{j=1}^{N} \lam_j$,
 $|\lam_N| \leq \|A\|$, and $|\lam_1| \geq \frac{1}{\|A^{-1}\|}$,
 the first two claim follows.
 For the third claim, use the first one to show that
 for any minor $\ti{A}$ of $A$, we have $|\det(A)| \leq \|A\|^{n-1}$.
\end{proof}

\begin{proof}[Proof of Theorem~\ref{thm:matcartan}]
 Define $\varphi(z) = \det(A(z))$. By Lemma~\ref{lem:detestis},
 we have that 
 $$
  \sup_{z \in \mathbb{D}^n} |\varphi(z)| \leq B^{N}\quad\text{and}\quad
  |\varphi(x_0)| \geq \frac{1}{D^N}.
 $$
 Hence, we can apply Proposition~\ref{prop:cartanbourgain}
 to the function
 $$
  f(z) = \frac{1}{B^N} \varphi(z + x_0).
 $$
 Since $2\E+\frac{1}{2} < 6$, we obtain $|f(z)| \leq 1$ 
 for $z \in (\mathbb{D}_{2\E})^n$,
 and $|f(0)| \geq \eps = \left(\frac{1}{B\cdot D}\right)^N$.
 By Proposition~\ref{prop:cartanbourgain}, we obtain
 $$
  |\{x \in [-\frac{1}{2}, \frac{1}{2}]^n:\quad |f(x)| \leq \E^{-s + 2 N \log(B)}\}|
    \leq 60 \E^3 n^{3/2} 2^n \exp\left(-\frac{s - 2 N \log(B)}{N \log(B \cdot D)}\right).
 $$
 By the previous lemma and $B \geq N$, we have
 $$
  \|A(x)^{-1}\| \leq \frac{B^{N}}{|\det(A(x))|}
   = \frac{B^{2N}}{|f(x)|}.
 $$
 The claim follows by some computations.
\end{proof}

Let us now discuss, what happens if we
replace the Lebesgue measure $| . |$ by an absolutely continuous
measure with density $\rho$. First note, that then the measure
in the case $n = 1$ of a set is just
\be
 \mu(A) = \int_{A} \rho(x) dx.
\ee
In particular, when the density $\rho$ is bounded, and supported
in $[-\frac{1}{2}, \frac{1}{2}]$, we obtain the simple estimate
\be
 \mu(A) \leq \| \rho \|_{\infty} \cdot |A|.
\ee
Of course, if we do this now for $A \subseteq [-\frac{1}{2}, \frac{1}{2}]^n$,
we obtain
\be
 \mu^{\otimes n} (A) \leq (\| \rho \|_{\infty})^n \cdot |A|.
\ee 
Hence, we see that our results remain valid.

\begin{theorem}\label{thm:matcartan2}
 Let $B \geq N$, $D > 0$, $s>0$,
 $A: \mathbb{D}^n \to \C^{N \times N}$ be an analytic function,
 and $\mu$ an absolutely continuous measure with bounded density $\rho$
 and $\operatorname{supp}(\mu)\subseteq [-\frac{1}{2},\frac{1}{2}$.
 Assume there exists $x_0 \in [-\frac{1}{2}, \frac{1}{2}]^n$
 such that
 \be
  \sup_{z \in \mathbb{D}^n} \|A(z) \|\leq B,\quad\text{and}\quad \|A(x_0)^{-1}\| \leq D .
 \ee
 Assume 
 \be
  \frac{s}{N} \geq 25 n \log(\|\rho\|_{\infty}) \cdot \log(B \cdot D).
 \ee
 Then
 \be
  \mu^{\otimes n}(\{x \in [-\frac{1}{2}, \frac{1}{2}]^n:\quad \|A(x)\| \geq \E^{s} \})
    \leq \exp\left(-\frac{1}{4} \frac{s}{N \log(B D)} \right).
 \ee
\end{theorem}

%
%
%

\section{Cartan's lemma for Schr\"odinger operators}
\label{sec:cartanschroe}

In this section, we will state a version of the results of the last section
suited for our application. In particular, we will combine them with
the results of Section~\ref{sec:resolv1}.
We will denote by $\mathcal{B}(\ell^2(\Lambda_R(0)))$ the Banach space
of Schr\"odinger operators $\ell^2(\Lambda_R(0)) \to \ell^2(\Lambda_R(0))$
with not necessarily real potential.
Given $A \in \mathcal{B}(\ell^2(\Lambda_R(0)))$, we denote by
$A^{\Xi}$ the restriction of $A$ to $\ell^2(\Xi)$, where $\Xi \subseteq\Lambda_R(0)$.
Similar results can be found in the work of Bourgain,
see for example Lemma~2 in \cite{b2009}.
We recall $\D =\{z:\quad |z|<6\}$.

\begin{theorem}\label{thm:schroecartan}
 Let
 \be
  H: \D^n \to \mathcal{B}(\ell^2(\Lambda_R(0)))
 \ee
 be an analytic function taking values in the normal operators.
 Given $\Lambda_{r_j}(m_j) \subseteq \Lambda_R(0)$ for $1 \leq j \leq J$.
 Assume
 \begin{enumerate}
  \item The set
   \be
    \Xi = \Lambda_R(0) \setminus \left(\bigcup_{j=1}^{J} \Lambda_{r_j}(m_j)\right)
   \ee
   is $r$-acceptable. 
  \item For $i \neq j$
   \be
    \Lambda_{r_i + 2 r}(m_i) \cap \Lambda_{r_j + 2 r}(m_j) = \emptyset.
   \ee
  \item The bound
   \be
    \sup_{z\in\mathbb{D}^n} \|H(z)\| \leq \E^{r^{\tau}}.
   \ee
  \item For any subcube $\Lambda_r(n)\subseteq\Xi$ and $z \in \D^n$
   \be
    \Lambda_r(n)\text{ is $(\gamma,\tau,1)$-suitable for } H(z).
   \ee
  \item There exists $x_0 \in \left[-\frac{1}{2},\frac{1}{2}\right]^n$
   such that for $1 \leq j \leq J$
   \be
    \|H^{\Lambda_{r_j + r}(m_j)}(x_0)^{-1}\| \leq \E^{(r_j)^{\tau}}.
   \ee
  \item The measure $\mu$ is absolutely continuous with bounded density $\rho$.
  \item Let $r_{\infty} = \max_{1\leq j\leq J} r_j$. 
   Let $S \geq T > 0$. The inequalities
   \begin{align}\label{eq:condupRlarge}
    \frac{S}{T} &\geq 1200 \cdot 3^d J (r_{\infty})^{d + 2\tau},\\
    S &\geq \max\left(10000 J 3^{d} (r_{\infty})^{d + 2 \tau} n \log(\|\rho\|_{\infty}),
     2(p+4)\log(2) + 10 r^{\tau}\right).
   \end{align}
 \end{enumerate}
 Then
 \be
  \mu^{\otimes n}\left(\left\{x \in \left[-\frac{1}{2},\frac{1}{2}\right]^n:\quad \|H(x)\| 
  \geq \frac{1}{2^{p}} \E^{S}\right\}\right)
   \leq \E^{-T}.
 \ee
\end{theorem}

In our applications, we will have $n \approx \#(\Xi)$. This is the big difference
to Schr\"odinger operators with quasi-periodic or skew-shift potential,
where $n$ is a small fixed number.

We will now begin with the setup of the proof. As
usually, it takes some notation.
Introduce 
\be
 \Theta = \bigcup_{j=1}^{J} \Lambda_{r_j}(m_j).
\ee
We note that $\Lambda_R(0) = \Xi \cup \Theta$. By the results
of Section~\ref{sec:resolv1}, we can conclude that

\begin{lemma}
 We have that
 \begin{enumerate}
  \item For $z \in \D^n$, we have
   \be
    \| (H^{\Xi}(z))^{-1} \| \leq \E^{3 r^{\tau}}.
   \ee
  \item We have
   \be
    \| (H(x_0))^{-1} \| \leq \E^{5 (r_{\infty})^{\tau}}.
   \ee
 \end{enumerate}
\end{lemma}

\begin{proof}
 These are consequences of Theorem~\ref{thm:resolv1}
 and Theorem~\ref{thm:resolv1b}.
\end{proof}

We now write our operator as a block matrix
\be
 H(z) = \begin{pmatrix} H^{\Xi}(z) & \Gamma_1(z) \\ \Gamma_2(z) & H^{\Theta}(z) \end{pmatrix}.
\ee
From (iii), we have that
\be
 \|H^{\Xi}(z)\|, \|\Gamma_1(z)\|, \|\Gamma_2(z)\|, \|H^{\Theta}(z)\| \leq \E^{r^{\tau}},
\ee
for $z \in \mathbb{D}^n$.
We now recall the Schur complement formula.

\begin{lemma}[Schur complement]\label{lem:schur}
 Assume $A$ is invertible. Then
 \be
  \begin{pmatrix} A & B \\ C & D \end{pmatrix},
 \ee
 is invertible, if and only if
 \be
  S = D - C A^{-1} B
 \ee
 is invertible. Furthermore then
 \be
  \begin{pmatrix} A & B \\ C & D \end{pmatrix}^{-1} =
  \begin{pmatrix} A^{-1} + A^{-1} B S^{-1} C A^{-1} & - A^{-1} B S^{-1} \\ - S^{-1} C A^{-1} & S^{-1} \end{pmatrix}.
 \ee
\end{lemma}

In particular, if we have
\be\label{eq:bddschur1}
 \|S^{-1}\| \leq \left\| \begin{pmatrix} A & B \\ C & D \end{pmatrix}^{-1} \right\|
\ee
and
\be\label{eq:bddschur2}
 \left\| \begin{pmatrix} A & B \\ C & D \end{pmatrix}^{-1} \right\| \leq
  (1 + \|S^{-1}\|)(1 + \|A^{-1}\|)^2(1 + \|B\|) (1 + \|C\|).
\ee
We now switch to out explicit setting with
\be
 A = H^{\Xi}(z),\quad B = \Gamma_1(z), \quad C = \Gamma_2(z), \quad D = H^{\Theta}(z).
\ee
We then have
\be
 S(z) = H^{\Theta}(z)  - \Gamma_2(z) (H^{\Xi}(z))^{-1} \Gamma_1(z).
\ee
By (iii), we have that 
\be
 \sup_{z \in \mathbb{D}^n} \|S(z)\| \leq 2 \E^{5 r^{\tau}},
\ee
and from \eqref{eq:bddschur1} that $\|S(x_0)^{-1}\| \leq \E^{5 (r_{\infty})^{\tau}}$.

\begin{proof}[Proof of Theorem~\ref{thm:schroecartan}]
 We have that $S(z)$ satisfies the assumptions of
 Theorem~\ref{thm:matcartan2}. We apply it with 
 $$
  s = S -5 r^{\tau} - (p+4)\log(2),
 $$
 $N \leq 3^d J (r_{\infty})^{d}$, $B \leq 2 \E^{5 (r)^{\tau}}$,
 and $D \leq \E^{5 (r_{\infty})^{\tau}}$. We obtain a set 
 $$
  \mathcal{B} \subseteq \left[-\frac{1}{2},\frac{1}{2}\right]^n
 $$
 satisfying 
 $$
  \mu^{\otimes n}(\mathcal{B}) \leq \exp\left(-\frac{1}{4} \frac{s}{N \log(B\cdot D)}\right).
 $$
 such that for $x\notin\mathcal{B}$, we have
 $$
  \|S(x)^{-1}\| \leq \frac{1}{16} \E^{-5r^{\tau}} \frac{1}{2^{p}} \E^{S}
 $$
 and
 By \eqref{eq:bddschur2}, we have
 $$
  \|H(x)\| \leq 16 \E^{5 r^{\tau}} \cdot \|S(x)\|.
 $$
 The claim follows.
\end{proof}


%
%
%

\section{Further probabilistic estimates and combinatorial results}
\label{sec:combinatoric2}

In this section, we return to result specific for the
operator $H_{\lam,\omega}$. We will first prove a variant of
the probabilistic estimate, Proposition~\ref{prop:probatmostonebadcube},
which will allow to obtain \eqref{eq:asthmresolv1b}.
Second, we will improve on Proposition~\ref{prop:combinatoric}
in order to obtain the geometric conditions from Theorem~\ref{thm:resolv1b}.
After this section, we will have to combine all our results to
prove the multi-scale step.

We now prove the second probabilistic estimate, which follows ideas from
Bourgain's work \cite{b2009}.

\begin{proposition}\label{prop:probmodifyomega}
 Given $s_1 \geq -1$, $t_1 \geq \max(r + \ol{r}, s_1 + \ol{r})$,
 and for $2 \leq k \leq K$
 \be\label{eq:defsktk}
  s_k \geq \ol{t_{k-1}},\quad t_k \geq s_k + \ol{r}.
 \ee
 Assume \eqref{eq:asfr}, \eqref{eq:asrgamp}, and
 for $r \leq u \leq t_K$ that
 \be\label{eq:probnotsuitableleqeps2}
  \mathbb{P}(\Lambda_{u}(0)\text{ is not $(\gamma,\tau,p)$-suitable for } H_{\lam,\omega} - E) 
   \leq \eps.
 \ee
 Then there exists a set $\mathcal{B}_{2}^{R,K}$ with the following properties.
 \begin{enumerate}
  \item $\mathbb{P}(\mathcal{B}_{2}^{R,K}) \leq (3 R)^d \eps^K$.
  \item Let $\omega \notin \mathcal{B}_{2}^{R}$ and $n \in \Lambda_R(0)$.
   There exists $1 \leq k \leq K$ such that for some
   \be
    \hat{\omega} = \omega \pmod{\Lambda_{s_k}^{R, t_{k-1}} (n)}
   \ee
   we have
   \be
    \Lambda_{t_k}^{K}(n)\text{ is $(\gamma,\tau,p-1)$-suitable for } H_{\lam,\hat{\omega}} - E.
   \ee
 \end{enumerate}
\end{proposition}

The proof is similar to the one of Proposition~\ref{prop:probatmostonebadcube}.
For $n \in \Lambda_R(0)$, introduce the set $\mathcal{B}^n$ such that
(ii) fails.

\begin{lemma}
 Assume \eqref{eq:probnotsuitableleqeps2} and let $n \in \Lambda_R(0)$.
 Then we have $\mathbb{P}(\mathcal{B}^n) \leq \eps^{K}$.
\end{lemma}

\begin{proof}
 We first discuss the statement of the analog of Lemma~\ref{lem:propXn}
 in this context.
 Let $1\leq k\leq K$. Introduce $Y^{k,q}$ as the set of $\omega$
 such that for every $\hat{\omega} = \omega \pmod{\Lambda_{s_k}^{R,t_{k-1}}(n)}$,
 we have
 $$
  \Lambda_{t_k}^{R}(n)\text{ is not $(\gamma,\tau,q)$-suitable for } H_{\lam,\hat{\omega}} - E.
 $$
 By \eqref{eq:probnotsuitableleqeps2}, we have $\mathbb{P}(Y^{k,p}) \leq \eps$.
 Denote by $X^{k}$ the set of $\omega$ such that
 there exists $\ti{\omega} = \omega \pmod{\Lambda_{t_k}(n)^{c}}$ such
 that 
 $$
  \ti{\omega} \in Y^{k, p-1}.
 $$ 
 One should note that $\omega \in X^k$ and
 $$
  \ti{\omega} = \omega \pmod{\Lambda_{s_k}^{R,t_{k-1}}(n) 
   \cup \Lambda_{t_k}^{R}(n)^{c}}
 $$
 then also $\ti{\omega} \in X^k$.

 By Lemma~\ref{lem:LamrstoLamr}, we have that $X^k \subseteq Y^{k,p}$,
 which implies $\mathbb{P}(X^k) \leq \eps$.
 As in Lemma~\ref{lem:propXn}, one checks that that $X^{k}$ and $X^{\ell}$
 are independent for $k \neq \ell$. 

 We also have
 $$
  \mathcal{B}^{n} \subseteq \bigcap_{k=1}^{K} X^k,
 $$
 which implies the claim.
\end{proof}

\begin{proof}[Proof of Proposition~\ref{prop:probmodifyomega}]
 Define
 $$
  \mathcal{B}_{2}^{R} = \bigcup_{n \in \Lambda_R(0)} \mathcal{B}^n.
 $$
 From the definition of $\mathcal{B}^n$, we have that (ii) holds.
 By the previous lemma (i) follows.
\end{proof}

We are now done with the probabilistic part of this section,
and now move on to the combinatorics. We will show a variant
of Proposition~\ref{prop:combinatoric}, which also allows
us to prove the geometric conditions of Theorem~\ref{thm:resolv1b}.

The main difference is that we will not only ensure that the sets
$\Lambda_{s_{q_j}}(m_j)$ are disjoint, but we will also wish to ensure
that the set
\be
 \Lambda_R(0) \setminus \left(\bigcup_{j=1}^{J} \Lambda_{\ti{r}_{j}}(m_j) \right)
\ee
is $r$-acceptable for any choice $r_{q_j} \leq \ti{r}_{j} \leq s_{q_j} - 2r - 1$.

\begin{theorem}\label{thm:combinatoric}
 Let $Q \geq (d + 1) K + 1$ and $R \geq r \geq 1$. 
 Given $K$ points $n_1, \dots, n_K \in \Lambda_R(0)$ and 
 $Q$ length scales $r \leq r_1 \leq s_1 \leq \dots \leq s_Q \leq r_Q \leq R$ satisfying
 \eqref{eq:condrq1rq} and
 \be
  s_q \geq r_q + 3 r.
 \ee 
 Then there exists a sequence of points $m_1,\dots,m_J$ and 
 numbers $q_1,\dots,q_J$ with $J\leq K$ such that (i) and (ii) from
 Proposition~\ref{prop:combinatoric} and 
 \begin{enumerate}
  \item[(iii)] For any choice $r_{q_j} \leq \ti{r}_{j} \leq s_{q_j} - 2r - 2$,
   we have that the set
   \be
    \Lambda_R(0) \setminus \left(\bigcup_{j=1}^{J} \Lambda_{\ti{r}_{j}}^{R}(m_j) \right)
   \ee
   is $r$-acceptable.
 \end{enumerate}
\end{theorem}

In order to prove this theorem, we prove two preliminary 
lemmas, which study properties of being $r$-acceptable
from Definition~\ref{def:setacceptable}.

\begin{lemma}
 Let $1 \leq k \leq K$. There exists a set $\mathcal{Q}^k$
 such that
 \begin{enumerate}
  \item $\#(\mathcal{Q}^k) \leq d$.
  \item For $q \notin \mathcal{Q}^k$ and $r_q \leq \ti{r} \leq s_q - r -1$,
   we have that
  \be
   \Lambda_R(0) \setminus \Lambda_{\ti{r}}^R(n)
  \ee
  is $r$-acceptable.
 \end{enumerate}
\end{lemma}

\begin{proof}
 Choose $\mathcal{Q}^k$ as the set of $q$ such that for some $1 \leq j \leq d$
 $$
  |n_j| + r_q \leq R \leq |n_j| + s_q.
 $$
 Clearly we have at most $d$ such choices, and if neither of
 the two conditions is the case, then the above set is
 $r$-acceptable.
\end{proof}

\begin{lemma}
 Suppose that
 \be\label{eq:LamRminI1}
  \Lambda_R(0) \setminus \left(\bigcup_{i=1 }^{I - 1}\Lambda_{r_i}^R(n_i)\right)
 \ee
 and
 \be\label{eq:LamRminI}
  \Lambda_R(0) \setminus \Lambda_{r_I}^{R}(n_I).
 \ee
 are $r$-acceptable. Furthermore suppose that for $i \neq I$
 \be
  \Lambda_{r_i + 2r+1}^{R, r_i}(n_i) \cap \Lambda_{r_I + 2r+1}^{R, r_I}(n_I) = \emptyset.
 \ee
 Then 
 \be
  \Lambda_R(0) \setminus \left(\bigcup_{i=1 }^{I}\Lambda_{r_i}^R(n_i)\right)
 \ee
 is $r$-acceptable.
\end{lemma}

\begin{proof}
 Let 
 $$
  \Xi = \Lambda_R(0) \setminus \left(\bigcup_{i=1 }^{I}\Lambda_{r_i}^R(n_i)\right).
 $$
 Let $x \in \Lambda_{r_I + 2r+1}^{R, r_I}(n_I) \cap \Xi$.
 By \eqref{eq:LamRminI}, we can find $\Lambda_r(n) \subset \Xi$
 such that 
 $$
  \dist(x,\partial^{\Xi}\Lambda_r(n))\geq\frac{r}{10}.
 $$
 For $x \in \Xi \setminus \Lambda_{r_I + 2r+1}^{R, r_I}(n_I)$,
 the claim follows from \eqref{eq:LamRminI1}.
\end{proof}

\begin{proof}[Proof of Theorem~\ref{thm:combinatoric}]
 By the first lemma, we can eliminate $Q \cdot K$ choices of the
 $r_q, s_q$, so that the individual cubes would all satisfy the
 $r$-acceptability assumption.

 By Proposition~\ref{prop:combinatoric}, we can now find a choice
 $m_1, \dots, m_J$ and $q_1, \dots, q_J$ such that (i) and (ii)
 hold. Then, by the previous lemma, we have that (iii) holds.
\end{proof}

%
%
%

\section{Setup for the proof of Theorem~\ref{thm:indstep}}

In this section, we will begin the proof of Theorem~\ref{thm:indstep}
and for this reason no longer work in full generality. In particular,
we will begin specializing to
\be
 p = 3,\quad \tau= \frac{1}{2}.
\ee
The main reason is that this way the numerical inequalities become
somewhat more transparent. We will furthermore define
\be
 \widehat{t} = \max(\ol{t}, 3 t, t + 3r),
\ee
which is motivated by the conditions of Theorem~\ref{thm:combinatoric}.

\begin{theorem}\label{thm:indsetup}
 Let $K \geq 1$, $r \geq 1$ and assume
 \be\label{eq:condKsetup}
  r \geq \left(\max(4,2 + \frac{4\gamma}{c})\right)^{2 d K^2}.
 \ee
 Assume for $r \leq u \leq r^3$ that
 \be\label{eq:probnotsuitableleqeps5}
  \mathbb{P}(\Lambda_{u}(0)\text{ is not $(\gamma,\frac{1}{2},3)$-suitable for }
    H_{\lam,\omega} - E) \leq \eps.
 \ee
 Then for $R \geq r^4$ there exists a set $\mathcal{B}^{R}$ 
 satisfying
 \begin{enumerate}
  \item We have
   \be\label{eq:probindsetup}
    \mathbb{P}(\mathcal{B}^{R}) \leq 3^d \left(\frac{3^{d K}}{(K+1)!} + 2 d K\right) R^{d K} \eps^{K}.
   \ee
  \item For each $\omega\notin \mathcal{B}^{R}$ there exist $0 \leq L \leq K$,
   $\ul{m} \in \Lambda_R(0)^{L}$, $s_{\ell} \geq 0$, and 
   \be
    \ol{s_{\ell}}  \leq t_{\ell} \leq r^3
   \ee
   such that the following hold.
   \begin{enumerate}
    \item The set
     \be
      \Lambda_R(0) \setminus \left(\bigcup_{\ell=1}^{L} \Lambda_{s_{\ell} + r}^{R}(m_{\ell})\right)
     \ee
     is $r$-acceptable.
    \item For $k \neq \ell$, we have
     \be
      \Lambda_{\widehat{t_{k}}}(m_{k}) \cap \Lambda_{\widehat{t_{\ell}}}(m_{\ell}) = \emptyset.
     \ee
    \item For
     \be
      \Lambda_r(n) \subseteq \Lambda_R(0) \setminus 
       \left(\bigcup_{\ell=1}^{L} \Lambda_{s_{\ell}}^{R}(m_{\ell})\right)
     \ee
     we have
     \be
      \Lambda_r(n)\text{ is $(\gamma,\frac{1}{2},2)$-suitable for } H_{\lam,\omega} - E.
     \ee
    \item For $1 \leq \ell \leq L$, there exists $0 \leq \ti{t} \leq s_{\ell}$
     \be
      \hat{\omega}_{\ell} = \omega \pmod{\Lambda_{s_{\ell}}^{R, \ti{t}}(m_{\ell})}
     \ee
     such that
     \be
      \|(H_{\lam,\hat{\omega}_{\ell}}^{\Lambda_{t_{\ell}}^{R}(m_{\ell})} - E)^{-1}\| 
        \leq \frac{1}{2} \E^{ \sqrt{t_{\ell}}}.
     \ee
   \end{enumerate}
 \end{enumerate}
 Furthermore the possible number of choices $s_q, t_q$ in (ii)
 is bounded by $r$.
\end{theorem}

The conditions in (ii) are chosen in such a way that the ones
of Theorem~\ref{thm:resolv1b} hold.

The proof of this theorem will proceed by combining the results of
the last section with the ones from Section~\ref{sec:suitability}.
Denote by $\mathcal{B}_{1}^{R,K}$ the set from
Proposition~\ref{prop:probatmostonebadcube}.
We see that (ii.c) holds for $\omega\notin\mathcal{B}_{1}^{R,K}$
as long as for $\ul{m}_{\omega}$, we have
\be
 \bigcup_{k=1}^{K} \Lambda_r(m_{\omega}^{k}) \subseteq \bigcup_{q=1}^{Q} \Lambda_{s_q}^{R}(m_q).
\ee
Next, we wish to apply Proposition~\ref{prop:probmodifyomega}.
Introduce 
\be
 Q = (d+1)K + 1,\quad s_1^1 = \ol{r}
\ee
and
\be
 t_k^{q} = \ol{s_{k}^{q}},\quad
 s_k^{q} = \ol{t_{k-1}^{q}},\quad
 s_{1}^{q} = \widehat{t_{K}^{q-1}}.
\ee
We will write $\ul{t}^{q} = \{t_k^{q}\}_{k=1}^{K}$
and $\ul{s}^{q} = \{s_k^{q}\}_{k=1}^{K}$. These are choosen
such that Theorem~\ref{thm:resolv1b} will be applicable.

\begin{lemma}\label{lem:estitinf}
 The number of $t_{k}^{q}$ is bounded by $2d K^2$.
 Furthermore, we have that
 \be\label{eq:boundtKQ}
  t_{K}^{Q} \leq r^3.
 \ee
\end{lemma}

\begin{proof}
 The first claim follows by the number being bounded by $Q K$.
 For the second claim, we have by Lemma~\ref{lem:proprbar}
 $$
 s_k^{q} \leq T_0 t_{k-1}^{q},\quad
 s_{1}^{q} \leq T_0 t_{K}^{q-1},\quad
 t_{k}^{q} \leq T_0 s_{k}^{q},
 $$
 as long as the terms on the right hand side are greater $\frac{\log(\lam)}{c}$.
 The claim follows since $\ol{r} \leq r^2$.
\end{proof}

Denote by $\mathcal{B}_{2,q}^{R,K}$
the set resulting by applying Proposition~\ref{prop:probmodifyomega}
to $\ul{s}^{q}$ and $\ul{t}^q$.
We then introduce
\be
 \mathcal{B}^{R} = \mathcal{B}_{1}^{R,K} \cup
  \left(\bigcup_{q=1}^{Q} \mathcal{B}_{2,q}^{R,K} \right).
\ee
One can easily check that (i) holds. 

Let now $\omega \notin \mathcal{B}^{R}$.
We have already seen that (ii.c) holds. Let's now check (ii.a) and (ii.b).
For this, we introduce
\be
 \hat{r}_q = s_1^{q},\quad \hat{s}_q = t_K^{q}.
\ee
Now we can apply Theorem~\ref{thm:combinatoric} apply to the
sequence $m_{\omega}^{k}$ and the $\hat{r}_q$ and $\hat{s}_q$.
We see that (ii.a) and (ii.b) hold as long as we choose
\be
 s_{\ell} \in \ul{s}^{q},\quad t_{\ell} \in \ul{t}^{q}.
\ee
Since $\omega \notin \mathcal{B}_{2,q}^{R,K}$, there is a
choice of $k$ such that for
\be
 s_{\ell} = s^{q}_{k}, \quad \ti{t} = t^{q}_{k-1},\quad t_{\ell} = t^{q}_{k},
\ee
we have that (ii.d) holds. This finishes
the proof of Theorem~\ref{thm:indsetup}.

%
%
%

\section{Application of Cartan's lemma}
\label{sec:appcartan}

The idea of this section is to study the measure of the set of $\omega$
such that conclusions (ii.a) to (ii.d) of Theorem~\ref{thm:indsetup}
hold but
\be
 \Lambda_R(0)\text{ is not $(\hat{\gamma},\frac{1}{2},3)$-suitable for } H_{\lam,\omega} - E
\ee
for some fixed $\hat{\gamma}$. In order to do this, we will show that
the Schr\"odinger valued Cartan Theorem, Theorem~\ref{thm:schroecartan},
implies that the measure of $\omega$, such that the necessary resolvent
estimates to apply Theorem~\ref{thm:resolv2} do not hold, is small.

\begin{definition}
 We see that $\ul{m},\ul{s},\ti{\ul{t}},\ul{t}$ obey the 
 geometric conditions for $\Lambda_R(0)$, if the following hold:
 \begin{enumerate}
  \item For $1 \leq \ell \leq L$, we have $m_{\ell} \in \Lambda_R(0)$ and
   $0 \leq \ti{t}_{\ell} \leq s_{\ell} \leq t_{\ell}$. 
  \item We have
   \be
    \ol{s_k} \leq t_k \leq r^3.
   \ee
  \item The set
   \be
    \Xi = \Lambda_R(0) \setminus \left(\bigcup_{\ell=1}^{L} \Lambda_{t_{\ell}}^{R} (m_{\ell}) \right)
   \ee
   is $r$-acceptable.
  \item For $k \neq \ell$
   \be
    \Lambda_{\widehat{t_{k}}}(m_{k}) \cap \Lambda_{\widehat{t_{\ell}}}(m_{\ell}) = \emptyset.
   \ee 
 \end{enumerate}
\end{definition}

These purely geometric conditions correspond to (ii.a) and (ii.b)
of Theorem~\ref{thm:indsetup}. We also note that (i) and (iii) ensure
the conditions (ii) and (iv) of Theorem~\ref{thm:resolv2}, and that
(i) - (iv) also imply the conditions (i), (ii), and (iii)
of Theorem~\ref{thm:resolv1b}.

\bigskip

We will now turn to introduce the conditions on the operator $H_{\lam,\omega}$,
which will depend on the random parameter $\omega$.
Let $\ul{m},\ul{s},\ti{\ul{t}},\ul{t}$ obey geometric conditions for $\Lambda_R(0)$.
Introduce the set $\Omega_{C}^{\ul{m}, \ul{s}, \ul{\ti{t}}, \ul{t}}$
as the set of $\omega$ satisfying
\begin{enumerate}
 \item For
  \be
   \Lambda_r(n) \subseteq \Lambda_R(0) \setminus 
    \left(\bigcup_{\ell=1}^{L} \Lambda_{s_{\ell}}^{R} (m_{\ell}) \right),
  \ee
  we have
  \be
   \Lambda_r(n)\text{ is $(\gamma,\frac{1}{2},3)$-suitable for } H_{\lam,\omega} - E.
  \ee
 \item For $1 \leq \ell \leq L$ there exists
  \be
   \hat{\omega}_{\ell} = \omega \pmod{\Lambda_{s_{\ell}}^{R,\ti{t}_{\ell}}(m_{\ell})}
  \ee
  such that
  \be
   \|(H_{\lam,\hat{\omega}_{\ell}}^{\Lambda_{t_{\ell}}^{R}(m_{\ell})} - E)^{-1}\|
    \leq \frac{1}{2} \E^{\sqrt{t_{\ell}}}.
  \ee
\end{enumerate}

\begin{remark}
 Theorem~\ref{thm:indsetup} implies that
 \be
  \mathbb{P}\left(\Omega\setminus\left(\bigcup_{\ul{m},\ul{s},\ul{\ti{t}},\ul{t}}\Omega^{\ul{m},\ul{s},\ul{\ti{t}},\ul{t}}_C\right)\right) \lesssim (\eps R)^{K}.
 \ee
\end{remark}

Define
\be
 t_{\infty} = \max_{1\leq\ell\leq L} t_{\ell}.
\ee
We are now ready for the main result of this section,
which will use the analyticity of the map $\omega\mapsto H_{\lam,\omega}$
for the first time.

\begin{theorem}\label{thm:applycartan}
 Assume Hypothesis~\ref{hyp:analytic}, $\gamma \geq 1$
 \begin{align}
  \label{eq:condrlargecartan}
   r &\geq \max(20000 \cdot K \cdot 3^d, \log(\|\rho\|_{\infty}), 2^{2 d K^2},
      (\log(2d+\lam))^2) \\
   \label{eq:condRlargecartan}
  R &\geq r^{3 d + 8} \\
   \label{eq:condrlargecartan2}
  t_{\infty} & \leq r^3.
 \end{align}
 There exists a set $\mathcal{B}^{\ul{m},\ul{s},\ul{\ti{t}},\ul{t}}_{C}$
 of measure 
 \be
   \mathbb{P}(\mathcal{B}^{\ul{m},\ul{s},\ul{\ti{t}},\ul{t}}_{C}) \leq \E^{- Kr}
 \ee
 such that for 
 \be
  \omega \in  \Omega^{\ul{m},\ul{s},\ul{\ti{t}},\ul{t}}_{C} 
   \setminus \mathcal{B}^{\ul{m},\ul{s},\ul{\ti{t}},\ul{t}}_{C}
 \ee
 we have
 \be
  \Lambda_R(0)\text{ is $(\hat{\gamma},\frac{1}{2},3)$-suitable for }H_{\lam,\omega} - E,
 \ee
 where $\hat{\gamma} = \gamma(1-\frac{2}{r})$.
\end{theorem}

We now proceed to give the prove of this theorem. For this, we will need
to switch from the probabilistic notions in statement to the more analytic
notions in Theorem~\ref{thm:schroecartan}.

We fix some $\omega_0 \in \Omega^{\ul{m},\ul{s},\ul{\ti{t}},\ul{t}}_{C}$.
Introduce
\be
 \mathcal{M} = \bigcup_{\ell=1}^{L} \Lambda_{s_{\ell}}^{R,\ti{t}_{\ell}}(m_{\ell}),\quad
 \ol{\mathcal{M}} = \bigcup_{\ell=1}^{L} \Lambda_{t_{\ell}}^{R}(m_{\ell}).
\ee
We introduce the restricted probability space $\widehat{\Omega}$ by
\be
 \widehat{\Omega} = \{\omega:\quad \omega=\omega_0\pmod{\mathcal{M}^c}\}.
\ee
In particular, we see that $\widehat{\Omega}$ is now $\#(\mathcal{M})$
dimensional. Using Fubini, we see that it is sufficient to prove the
following variant of the main theorem

\begin{proposition}\label{prop:reducedcartan}
 Assume $\gamma \geq 1$ \eqref{eq:condrlargecartan},
 \eqref{eq:condRlargecartan}, and \eqref{eq:condrlargecartan2}.
 There exists a set $\widehat{\mathcal{B}}$ of measure
 \be
  \mu^{\otimes \#(\mathcal{M})}(\widehat{\mathcal{B}}) \leq \E^{-K r}
 \ee
 such that for
 $\omega \in \widehat{\Omega} \setminus \widehat{\mathcal{B}}$,
 we have
 \be
  \Lambda_R(0)\text{ is $(\hat{\gamma},\frac{1}{2},3)$-suitable for }H_{\lam,\omega} - E,
 \ee
\end{proposition}

It will furthermore be convenient to identify
\be
 \widehat{\Omega} \cong \left[-\frac{1}{2}, \frac{1}{2}\right]^{\nu},
\ee
where
\be
 \nu = \#(\mathcal{M}) \leq 3^d L \left(\max_{1 \leq \ell \leq L} s_{\ell} \right)^d.
\ee
We now check the conditions of Theorem~\ref{thm:schroecartan}.

\begin{lemma}
 The map $\omega\mapsto H_{\lam,\omega} - E$ extends
 to analytic map
 \be
  \mathbb{D}^{\nu} \ni z \mapsto \mathcal{H}(z).
 \ee
 The following properties hold
 \begin{enumerate}
  \item The bound
   \be 
    \sup_{z \in \mathbb{D}^{\nu}} \|\mathcal{H}(z)\| \leq \E^{\sqrt{r}}.
   \ee
  \item For
   \be
    \Lambda_r(n) \subseteq \Lambda_R(0) \setminus \ol{\mathcal{M}}
   \ee
   we have
   \be
    \Lambda_r(n)\text{ is $(\gamma,\tau,0)$-suitable for } \mathcal{H}(z).
   \ee
 \end{enumerate}
\end{lemma}

\begin{proof}
 The existence of the analytic extension is a consequence of
 Hypothesis~\ref{hyp:analytic}. From \eqref{eq:asfr}, we can
 conclude that
 $$
  \sup_{z \in \mathbb{D}^{\nu}} \|\mathcal{H}(z)\| \leq 2d + \lam.
 $$
 (i) now follows from \eqref{eq:condrlargecartan}.

 Since $t_{\ell} \geq s_{\ell} + \ol{r} - r$, we have
 $$
  \Lambda_{\ol{r}}(n) \subseteq \Lambda_R(0)\setminus\mathcal{M}.
 $$
 Thus we have $z=\omega_0\pmod{\Lambda_{\ol{r}}(n)}$.
 So (ii) follows by Lemma~\ref{lem:LamrstoLamr}.
\end{proof}

\begin{lemma}
 There exists $\hat{\omega} \in \widehat{\Omega}$ such that for $1 \leq \ell \leq L$,
 we have
 \be
  \|(H^{\Lambda_{t_{\ell}}^{R}(m_{\ell})}_{\lam,\hat{\omega}} - E)^{-1}\|
   \leq \E^{\sqrt{t_{\ell}}}
 \ee
 or $\|\mathcal{H}^{\Lambda_{t_{\ell}}^{R}(m_{\ell})}(\hat{\omega})\| \leq \E^{\sqrt{t_{\ell}}}$.,
\end{lemma}

\begin{proof}
 Observe that for $k \neq \ell$, we have
 $$
  \Lambda_{\ol{t_{k}}}^{R,t_{k}}(m_{k}) \cap \Lambda_{\ol{t_{\ell}}}^{R,t_{\ell}}(m_{\ell})= \emptyset.
 $$
 Define $\hat{\omega}$ by
 $$
  \hat{\omega}_x = \begin{cases} (\hat{\omega}_{\ell})_x,& x \in \Lambda_{s_{\ell}}^{R,\ti{t}_{\ell}}(m_{\ell}); \\
   \omega_0,&\text{otherwise}.\end{cases}
 $$
 It is easy to see that $\hat{\omega} \in\widehat{\Omega}$.
 The claim now follows by Lemma~\ref{lem:LamrstoLamr}.
\end{proof}

We also obtain

\begin{lemma}\label{lem:cartesti2}
 Assume
 \be
  R^{\frac{1}{4}} \geq \max\left(25 \cdot 3^d K 
    \left(\max_{1\leq\ell\leq L}t_{\ell}\right)^{d+\frac{1}{2}},
   100 \log(\|\rho\|_{\infty}) \sqrt{\max_{1\leq\ell\leq L}t_{\ell}}\right).
 \ee
 We have
 \be
  \mu^{\otimes\nu}(\{\omega:\quad \|(H_{\lam,\omega}^{\Lambda_{R}(0)} - E)^{-1}\|
   > \frac{1}{8} \E^{\sqrt{R}}\}) \leq \E^{-R^{\frac{1}{4}}}.
 \ee
\end{lemma}

\begin{proof}
 We will apply Theorem~\ref{thm:schroecartan} here with 
 $$
  S = R^{\frac{1}{2}}, \quad T = R^{\frac{1}{4}}.
 $$
 Then the claim follows after some computations.
\end{proof}

We now proceed to prove the resolvent estimates on the cubes $\Lambda_{t_{\ell}}(m_{\ell})$.
The main difficulty is that in order to apply Theorem~\ref{thm:resolv2},
we will need to ensure \eqref{eq:resolvbddagamr} with some $a$ such that
$t_{\ell} - s_{\ell} \geq a r$. Define
\be
 a = r^{3d + 3}.
\ee

\begin{lemma}\label{lem:checkcondST}
 Assume $K \leq r$, $\gamma \geq 1$
 \begin{align}
  r &\geq \max(20000 \cdot K \cdot 2^d, \log(\|\rho\|_{\infty}), 2^{2 d K^2}) \\
  t_{\infty} & \leq r^3.
 \end{align}
 With $S = \gamma a r$ and $T = K r$, \eqref{eq:condupRlarge}
 holds.
\end{lemma}

\begin{proof}
 A computation shows that \eqref{eq:condupRlarge} is equivalent
 to 
 $$
  S \geq r^{3d + 4},\quad \frac{S}{T} \geq r^{3d + 2}.
 $$
 The claim follows since $T \leq r^2$.
\end{proof}

Let $Q = (d+1)K + 1$ and apply Theorem~\ref{thm:combinatoric} with
the $r$ from it being
\be
 t_{\infty} = \max_{1\leq\ell\leq L} t_{\ell}
\ee
and to the length scales
\be
 r_q = t_{\infty} + (2 q - 1) a r,\quad
 s_q = t_{\infty} + 2 q a r.
\ee
Denote the resulting choice by $\check{t}_{\ell}$ and 
$\check{m}_{\ell}$. We now come to

\begin{lemma}\label{lem:cartesti1}
 We have that
 \be
  \mu^{\otimes\nu}\left(\left\{\omega:\quad 
   \|(H_{\lam,\omega}^{\Lambda_{\check{t}_{\ell}}^{R}(\check{m}_{\ell})} - E)^{-1}\| > 
    \frac{1}{8}\E^{a \gamma r} \right\}\right) \leq \E^{-K\cdot r}.
 \ee
\end{lemma}

\begin{proof}
 Apply Theorem~\ref{thm:schroecartan} with 
 $$
  S = a \gamma r,\quad T = K\cdot r
 $$
 and the claim follows by Lemma~\ref{lem:checkcondST}.
\end{proof}

We are now ready for

\begin{proof}[Proof of Proposition~\ref{prop:reducedcartan}]
 First we see that using Lemma~\ref{lem:cartesti1} and \ref{lem:cartesti2},
 we need to eliminate  a set of measure
 $$
  L \cdot \E^{-r} + \E^{-R^{\frac{1}{4}}}
  \leq \E^{-\frac{1}{2} r}.
 $$
 Here, we used $r \geq \frac{1}{2} \log(2 K)$.

 We are now in a situation, we can apply Theorem~\ref{thm:resolv2}.
 We obtain that
 $$
  \frac{\hat{\gamma}}{\gamma} \geq 1 - \frac{1}{r}
   - \frac{10}{R}\left(2 K (t_{\infty} + K a r) + \sqrt{R} + d \log(3)\right).
 $$
 For $R \geq \max(160, 40 d \log(3)) r$, this reduces to
 $$
  \frac{\hat{\gamma}}{\gamma} \geq 1 - \frac{3}{2}\frac{1}{r}
   - \frac{10}{R}\left(2 K (t_{\infty} + K a r)\right).
 $$
 By Lemma~\ref{lem:estitinf}, this can be further reduced to
 $$
  \frac{\hat{\gamma}}{\gamma} \geq 1 - \frac{3}{2}\frac{1}{r}
   - \frac{1600K 2^{d K^2} a r}{R}.
 $$
 By assumptions $10000K 2^{d K^2} a r^2 \leq R$, it follows that
 $\hat{\gamma} \geq \gamma (1 -\frac{2}{r})$, which is what
 we wanted.
\end{proof}

%
%
%
%

\section{Proof of Theorem~\ref{thm:indstep}}
\label{sec:proofthmindstep}

Our goal in this section is to complete the proof
of Theorem~\ref{thm:indstep}. Before jumping into 
the proof, to take the time to check various inequalities.
We define
\be\label{eq:defKproof}
 K = \left\lfloor\sqrt{\frac{1}{2d} \frac{\log(r)}{\log(\max(4,2 + \frac{4\gamma}{c}))}}\right\rfloor
\ee
such that the last condition of \eqref{eq:condrlargecartan}
holds. We even obtain

\begin{lemma}
 Define $K$ by \eqref{eq:defKproof}. Assume
 \be
  r \geq \log(2^{200 d}, \|\rho\|_{\infty}).
 \ee
 Then \eqref{eq:condrlargecartan} holds.

 Furthermore, for $r \geq 2^{4 \alpha d}$ we have  $K\geq \alpha$
\end{lemma}

\begin{proof}
 The last two conditions hold by definition or assumption.
 From $r \geq 2^{200 d}$, we obtain by \eqref{eq:defKproof}
 that $K \geq 10$. Thus, we have 
 $$
  r \geq 2^{2dK^2} \geq 2^{20 K} \cdot 2^{2 K} \cdot 2^{2 d},
 $$
 since $K\geq 10$ and $d \geq 1$. This implies the first condition.

 The last claim is a computation.
\end{proof}

We collect the result in the following lemma

\begin{lemma}
 There exists a set $\mathcal{B}_{1}^{R}$ such that
 \be
  \mathbb{P}(\mathcal{B}^{R}_{1}) \leq \frac{1}{2} \frac{1}{R^{K}}
 \ee
 Furthermore the conclusions from (ii) of Theorem~\ref{thm:indsetup}
 hold for $\omega\notin\mathcal{B}^{R}_{1}$. 
\end{lemma}

\begin{proof}
 We can apply  Theorem~\ref{thm:indsetup} with $\eps = \frac{1}{(r_0)^{\alpha}}$.
 To obtain a set $\mathcal{B}^{R}_{1}$ satisfying
 $$
  \mathbb{P}(\mathcal{B}^{R}_{1}) \leq 3^d \left(\frac{3^{d K}}{(K+1)!} + 2 d K\right)
  \cdot \left(\frac{R^{d}}{(r_0)^{\alpha}}\right)^{K}
 $$
 By assumption and the choice of $R$, we have
 $$
  3^d \left(\frac{3^{d K}}{(K+1)!} + 2 d K\right) \frac{1}{(r_0)^K} \leq \frac{1}{2}, \quad
  \frac{R^{d}}{(r_0)^{\alpha - 1}} \leq \frac{1}{R}.
 $$
 The claim follows.
\end{proof}

Second, we will want to apply Theorem~\ref{thm:applycartan}.
For this, we first note

\begin{lemma}
 Denote the number of choices of $\ul{m}, \ul{s}, \ul{\ti{t}}, \ul{t}$ 
 in Theorem~\ref{thm:indsetup} by $\#(\text{choices})$. We have that
 \be
  \#(\text{choices}) \leq 2 d K^2 (3R)^{d K}.
 \ee
\end{lemma}

\begin{proof}
 The number of choices for $\ul{m}$ is bounded by $(3R)^{dK}$.
 The claim now follows from the last statement in Theorem~\ref{thm:indsetup}.
\end{proof}

Next, we apply Theorem~\ref{thm:applycartan}
to all possible choice of $\ul{m}, \ul{s}, \ul{\ti{t}}, \ul{t}$ to
obtain a set $\mathcal{B}^{R}_2$.

\begin{lemma}
 For $r \geq 1$ large enough, we have that
 \be
  \mathbb{P}(\mathcal{B}^{R}_2) \leq \frac{1}{2} \frac{1}{R^{K}}
 \ee
\end{lemma}

\begin{proof}
 This follows since $r$ is large enough to estimate
 $$
  2 d K^2 (3R)^{d K} \E^{-rK} \leq \frac{1}{2} \frac{1}{R^{K}}.
 $$
 This finishes the proof.
\end{proof}

We now finally come to

\begin{proof}[Proof of Theorem~\ref{thm:indstep}]
 Comparing the conclusions of Theorem~\ref{thm:indsetup} and
 Theorem~\ref{thm:applycartan}, we then obtain for
 $$
  \omega \notin \mathcal{B}^{R} = \mathcal{B}^{R}_{1} \cup \mathcal{B}^{R}_{2}
 $$
 that
 $$
  \Lambda_R(0)\text{ is $(\hat{\gamma},\frac{1}{2},3)$-suitable for }
   H_{\lam,\omega} - E,
 $$
 where $\hat{\gamma} = \gamma (1 - \frac{2}{r_0})$.

 Now again for $r$ large enough, we now have
 $$
  \mathbb{P}(\mathcal{B}^{R}) \leq \frac{1}{R^{K}}.
 $$
 This finishes the proof.
\end{proof}


%
%
%

\section{A first step towards localization}
\label{sec:firststeplocalization}

In this section, we begin to draw conclusions from the results
of multi-scale analysis. The main topic will be to draw conclusions
from the knowledge that the estimate 
\be
 \mathbb{P}(\Lambda_r(0)\text{ is not $(\gamma,\tau,2)$-suitable for }H_{\lam,\omega} - E)
  \leq \frac{1}{r^{4 d}}
\ee
holds for all energies $E$. In this section and the following
sections, we will address questions about the spectral type
of $H_{\lam,\omega}$ and the dynamics of the time evolution
$\E^{-\I t H_{\lam,\omega}}$. In Appendix~\ref{sec:wegner}, we discuss
what continuity property of the integrated density of states
this implies.

\bigskip

I will begin by introducing the concept of {\em generalized eigenfunction}.
We call a solution $\psi \neq 0$ of
\be\label{eq:genevsol}
 H_{\lam, \omega} \psi = E \psi
\ee
interpreted as a formal difference equation a {\em generalized eigenfunction},
if it obeys the growth condition
\be\label{eq:genevbdd}
 |\psi(n)| \leq (1 + |n|_{\infty})^{2 d}.
\ee
We call $E$ the {\em generalized eigenvalue}.
Given $\eps > 0$ and $R \geq 1$, we introduce $\mathcal{E}_{\lam,\omega}^{\eps,R}$
as the set of all generalized eigenvalues $E$ of $H_{\lam,\omega}$,
where the generalized eigenfunction $\psi$ obeys
\be\label{eq:genevnorm}
 \sum_{x \in \Lambda_R(0)} |\psi(x)|^2 \geq \eps.
\ee
We introduce the set $\mathcal{E}_{\lam,\omega}$ of all generalized
eigenvalues as
\be
 \mathcal{E}_{\lam,\omega} = \bigcup_{\eps > 0} \bigcup_{R \geq 1} 
  \mathcal{E}_{\lam,\omega}^{\eps,R}.
\ee
We furthermore, recall that given $\varphi \in \ell^2(\Z)$, the
associated spectral measure $\mu_{\lam,\omega}^{\varphi}$ is
characterized by
\be
 \spr{\varphi}{(H_{\lam,\omega} - z)^{-1} \varphi} =
  \int \frac{1}{t-z} d\mu_{\lam,\omega}^{\varphi} (t).
\ee
The importance of the generalized eigenfunctions comes from

\begin{proposition}
 For any $\varphi\in\ell^2(\Z)$, we have that
 \be
  \mu_{\lam,\omega}^{\varphi}(\R \setminus \mathcal{E}_{\lam,\omega}) = 0.
 \ee
 This means that all the spectral measures are supported
 on $\mathcal{E}_{\lam,\omega}$.
\end{proposition}

\begin{proof}
 This is Proposition~7.4 in \cite{kirsch}.
\end{proof}

With a slight abuse of notation, we will often write
$\psi \in \mathcal{E}_{\lam, \omega}^{\eps,R}$, if there exists
$E \in \mathcal{E}_{\lam, \omega}^{\eps,R}$ such that \eqref{eq:genevsol}
holds and $\psi$ obeys \eqref{eq:genevbdd} and \eqref{eq:genevbdd}.
The main reason is that most of the following statements
are concerned with the eigenfunctions and not the eigenvalues.

In particular, we obtain the following corollary

\begin{corollary}
 Suppose we can show for any $\psi \in \mathcal{E}_{\lam,\omega}$ that
 $\psi \in \ell^2(\Z^d)$. Then the spectrum of $H_{\lam,\omega}$
 is pure point.
\end{corollary}

We will now turn towards investigating the set $\mathcal{E}_{\lam,\omega}^{\eps,R}$
for fixed $\eps>0$ and $R\geq 1$. This understanding will be important
in order to be able to prove dynamical localization. The first result is

\begin{proposition}\label{prop:closetosigmafin}
 Assume for $r \geq 2^{k_1}$
 \be\label{eq:asprobleqr4d}
  \mathbb{P}(\Lambda_r(0)\text{ is not $(\gamma,\tau,2)$-suitable for }H_{\lam,\omega} - E)
   \leq \frac{1}{r^{4d}}.
 \ee
 Let $k \geq k_1$ and
 \be
  k \geq \frac{1}{\log(2)} \max\left(3 \gamma, \log\left(\frac{8d}{\gamma}\right), 
     \log\left(\frac{R}{2}\right)\right) + 1.
 \ee
 Then there exists a set $\mathcal{B}_k^{L}$ satisfying
 \begin{enumerate}
  \item The measure estimate
   \be 
    \mathbb{P}(\mathcal{B}_k^{L}) \leq 3 \lam \E^{-2^{k}}.
   \ee
  \item Let $r = 2^k$. For $\omega_0 \notin \mathcal{B}_k^{L}$, we have for
   \be
    \omega = \omega_0 \pmod{\Lambda_{4 r \ol{r}}(0)^{c}}
   \ee
   that for $\eps \geq \E^{-\frac{\gamma}{8} r}$
   \be\label{eq:disttosigma1}
    \dist(\mathcal{E}_{\lam,\omega}^{\eps,R}, \sigma(H_{\lam,\omega}^{\Lambda_{3 r \ol{r}}(0)}))
     \leq \E^{-\frac{\gamma}{8}r}.
   \ee
 \end{enumerate}
\end{proposition}

The proof of this proposition proceeds in several steps.
Consider for $t$ and $E$ the set $\mathcal{B}_{r}^{t}(E)$
of all $\omega$ such that for every
\be
 \ti{\omega} = \omega \pmod{ \Lambda_{t - \ol{r}}(0) \cup \Lambda_{t + \ol{r}}(0)^c}
\ee
we have that for every $n$ with $|n|_{\infty} = t$
\be
 \Lambda_r(n)\text{ is $(\gamma,\tau,1)$-suitable for }H_{\lam,\ti{\omega}} - E.
\ee
We have the following lemma

\begin{lemma}\label{lem:probannulus}
 Assume \eqref{eq:asprobleqr4d}.
 For $2d \leq t \leq r^{3}$, we have that
 \be
  \mathbb{P}(\mathcal{B}_{r}^{t}(E)^c) \leq \frac{1}{r^{d}}.
 \ee
 Furthermore, we have for 
 \be
  |t_1 - t_2| \geq 2 \ol{r} + 1
 \ee
 that $\mathcal{B}_{r}^{t_1}(E)$ and $\mathcal{B}_{r}^{t_2}(E)$ are independent.
\end{lemma}

\begin{proof}
 Denote by $X_n^r$ the set from Lemma~\ref{lem:propXn}.
 We have
 $$
  \mathcal{B}_{r}^{t} \subseteq \bigcap_{|n|_{\infty} = t} X_n^{r}.
 $$
 Since the number of such $t$ is bounded by $2 d t^{d-1} \leq r^{3 d}$,
 the claim follows.
\end{proof}

Define $t_j$ by
\be
 t_j = (1 + 2j) \ol{r}.
\ee
Introduce $\mathcal{B}_k (E)$ as the set of $\omega$ such that
for every $1 \leq j \leq r$, we have
\be 
 \omega \in \mathcal{B}_{r}^{t_j} (E).
\ee
Furthermore, we have $t_r + r + 1 \leq 3 r \ol{r}$.
Because of independence of the $\mathcal{B}_{r}^{t_j}(E)$, we have
that
\be
 \mathcal{B}_k(E) = \bigcap_{j=1}^{r} \mathcal{B}_{r}^{t_j}(E)
\ee
and thus by the previous lemma
\be
 \mathbb{P}(\mathcal{B}_k(E)) \leq \E^{-d \log(r) r}.
\ee
We obtain

\begin{lemma}\label{lem:probexistannulus}
 There exists a set $\mathcal{B}_k(E)$ with the following properties
 \begin{enumerate}
  \item $\mathbb{P}(\mathcal{B}_k(E)) \leq \E^{-d \log(r) r}$.
  \item For $\omega\notin \mathcal{B}_k(E)$  there exists $1 \leq j \leq r$
   such that for $|E - \ti{E}| \leq \E^{-3 \gamma r}$ and 
   \be
    \ti{\omega} = \omega \pmod{\Lambda_{4 r \ol{r}}(0)^c}
   \ee
   we have for $|n|_{\infty} = t_j$
   \be
    \Lambda_{r}(n)\text{ is $(\gamma,\tau,0)$-suitable for } H_{\lam,\ti{\omega}} - \ti{E}.
   \ee
 \end{enumerate}
\end{lemma}

\begin{proof}
 This follows from the discussion preceding the statement
 and Lemma~\ref{lem:perturbEsuitable} to perturb $E$. Here,
 we used that \eqref{eq:asrgamp} holds for $k$ large enough.
\end{proof}

Since $\sigma(H_{\lam,\omega}) \subseteq [-3\lam, 3\lam]$, we can
introduce
\be
 \mathcal{B}_k = \bigcup_{\ell = - \frac{3}{2} \lam \E^{3 \gamma r}}^{\frac{3}{2} \lam \E^{3 \gamma r}}
  \mathcal{B}_k(2 \ell \E^{-3 \gamma r}).
\ee
Then for $\omega \notin \mathcal{B}_k$ the conclusion (ii) of the
previous lemma holds for all $\ti{E}$. Furthermore, we have that
\be
 \mathbb{P}(\mathcal{B}_k) \leq 3 \lam \E^{(3 \gamma - d \log(r))r},
\ee
which is small as long as $d \log(r) > 3 \gamma$. Since $\log(r) = k \log(2)$,
we obtain that we must have
\be
 k \geq \frac{3\gamma}{\log(2)} + 1.
\ee
For the proof, we will furthermore need the elementary inequality
\be
 (1 + x)^{p} \leq \E^{\eps x}
\ee
for $x \geq \frac{2 p}{\eps}$.

\begin{proof}[Proof of Proposition~\ref{prop:closetosigmafin}]
 Let $E \in \mathcal{E}_{\omega}^{B,r}$ and $\omega \notin \mathcal{B}_k$
 constructed above. Then we can find $t$ such that for every $n$
 with $|n|_{\infty} = t$, we have
 $$
  \Lambda_r(n)\text{ is $(\gamma,\tau,0)$-suitable for } H_{\lam, \omega} - E.
 $$
 Hence, we obtain for these $n$ by \eqref{eq:unasgreen}
 and \eqref{eq:genevbdd} for $r \geq \frac{8 d}{\gamma}$ that
 $$
  |\psi(n)| \leq \E^{-\frac{\gamma}{2} r}.
 $$
 Consider the test function
 $$
  u(x) = \begin{cases} \psi(x), & x \in \Lambda_{t-1}(0);\\  0, &\text{otherwise},\end{cases}
 $$
 which satisfies $\|(H_{\lam,\omega}^{\Lambda_{3 r \ol{r}}(0)} - E) u\| \leq \E^{-\frac{\gamma}{4} r}$,
 by the choice of $r$. The claim follows since 
 $$
  \|u\|_{\ell^2(\Lambda_{R}(0))} \geq \E^{-\frac{\gamma}{8} r}
 $$
 by assumption.
\end{proof}

%
%
%

\section{Super polynomial decay of the eigenfunctions}

We will show 

\begin{theorem}\label{thm:ppsqrt}
 Assume \eqref{eq:asprobleqr4d}.
 There exists a set $\Omega_1$ and a constant $\hat{\gamma} > 0$ satisfying
 \begin{enumerate} 
  \item $\mathbb{P}(\Omega_1) = 1$.
  \item For $\omega\in\Omega_1$ the spectrum of $H_{\lam,\omega}$
   is pure point.
  \item For every $\omega\in\Omega_1$, there exists $\ell \geq 1$
   such that for $k \geq \ell$,
   $\|\psi\| = 1$ solving $H_{\lam,\omega} \psi = E \psi$
   with
   \be
    \sum_{x \in \Lambda_{2^{k -2}}(0)} |\psi(x)|^2 \geq \E^{-\frac{\gamma}{8} 2^{k}},
   \ee
   we have for $|n|_{\infty} \geq 2^{4 k}$ that
   \be
    |\psi(n)| \leq \E^{-c \sqrt{|n|_{\infty}}}.
   \ee
 \end{enumerate}
\end{theorem}

Let $\mathcal{B}^{L}_{k}$ be the set from Proposition~\ref{prop:closetosigmafin}.
Introduce
\be
 \mathcal{B}^{L} = \bigcap_{\ell \geq 1} \bigcup_{k \geq \ell} \mathcal{B}_{k}^{L}.
\ee
It follows that $\mathbb{P}(\mathcal{B}^{L}) = 0$.
Let now $\omega_0 \notin \mathcal{B}^{L}$. Then by Proposition~\ref{prop:closetosigmafin},
we have that there exists some $k_0({\omega})$ such that
for $k \geq k_0(\omega)$, we have with $r = 2^{k}$
for
\be
 \omega = \omega_0 \pmod{\Lambda_{4 r \ol{r}}(0)^{c}}
\ee
and $\eps \geq \E^{-\frac{\gamma}{8} r}$ that
\be
 \dist(\mathcal{E}_{\lam,\omega}^{\eps,B}, \sigma(H_{\lam,\omega_0}^{\Lambda_{3 r \ol{r}}(0)})) 
  \leq \E^{-\frac{\gamma}{8} r}.
\ee

Denote by $\mathcal{B}_{r}^{t}(E)$ the set constructed
before Lemma~\ref{lem:probannulus}. We introduce
\be
 \mathcal{B}_{r}^{t}(\omega_0) = 
  \bigcup_{E \in \sigma(H^{\Lambda_{3 r \ol{r}}(0)}_{\lam,\omega_0})}
   \mathcal{B}_{r}^{t}(E).
\ee
We introduce
\be
 \mathcal{B}_{k}^{S}(\omega_0) = \bigcup_{4 r \ol{r} \leq t \leq 16 r \ol{r}}
  \begin{cases} \mathcal{B}_{r}^{t}(\omega_0),& k \geq k_0(\omega_0); \\
   \emptyset, & \text{otherwise}. \end{cases}
\ee
We now define a set $\mathcal{B}_{k}^{S}$ as follows. Define
a map
\be
 [-\frac{1}{2},\frac{1}{2}]^{\Lambda_{4r\ol{r}}(0)} \to [-\frac{1}{2},\frac{1}{2}]^{\Z^d}
\ee
by mapping $x$ to some $\omega \in \mathcal{B}^{L}$ satisfying
\be
 \omega = x \pmod{\Lambda_{4r\ol{r}}(0)^c}
\ee
if such an $\omega$ exists, otherwise to any $\omega$
satisfying this condition. Then, we define $\mathcal{B}_{k}^{S}$
as the union over the set $\mathcal{B}_{k}^{S}(\omega)$ with
$\omega$ constructed above.

\begin{lemma}
 We have that
 \be
  \mathbb{P}(\mathcal{B}_{k}^{S}) \leq \E^{-2^{k - 2}}.
 \ee
\end{lemma}

\begin{proof}
 We have that $\mathcal{B}_{k}^{S}(\omega_0)$ is the intersection of
 less then
 $$
  2^{k+1} \cdot (3 2^{k+1})^{d} 
 $$
 many sets of measure $\leq \E^{-2^{k}}$. Hence
 $$
  \mathbb{P}(\mathcal{B}_{k}^{S}(\omega_0)) \leq \E^{-2^{k - 2}}.
 $$
 The claim now follows by Fubini.
\end{proof}

We can now introduce
\be
 \mathcal{B}^{S} = \bigcap_{\ell \geq 1}
  \left( \bigcup_{k\geq l} \mathcal{B}_{k}^{S} \right),
\ee
which satisfies $\mathbb{P}(\mathcal{B}^{S}) = 0$ by
a Borel--Cantelli argument.  Define
\be
 \Omega_1 = [-\frac{1}{2},\frac{1}{2}]^{\Z^d} 
  \setminus (\mathcal{B}^{S} \cup \mathcal{B}^{L}).
\ee
Assume now that $\psi \in \mathcal{E}_{\lam, \omega}^{R, \eps}$
for some $\omega \in \Omega_1$. Then there exists
$\ell \geq 1$ such that for $k \geq \ell$
\be
 \omega \notin (\mathcal{B}^{S}_{k} \cup \mathcal{B}^{L}_{k}).
\ee
In particular, we obtain that
\be
 |\psi(n)| \leq \E^{-c \sqrt{|n|_{\infty}}}
\ee
for some $c$ once $|n|_{\infty} \geq 2^{\ell}$.

\begin{proof}[Proof of Theorem~\ref{thm:ppsqrt}]
 It is easy to see that, we can conclude pure point spectrum,
 that is (i). Now (ii) follows after some computations.
\end{proof}

%
%
%

\section{Dynamical Localization}
\label{sec:dynloc}

In this section, we will adapt the machinery of the last section to
prove dynamical localization. The proof follows the strategy
of Bourgain and Jitomirskaya from \cite{bj}, where it was used
to prove dynamical localization for a certain quasi-periodic
band model.

Recall that $\{e_x\}_{x\in\Z^d}$
denotes the standard basis of $\ell^2(\Z)$, that is
\be
 e_x(n) = \begin{cases} 1, &x=n;\\ 0,& \text{otherwise}.\end{cases}
\ee
For simplicity, we will consider the Schr\"odinger
equation with initial condition $e_0$, that is
\be\label{eq:schroeevol}\begin{split}
 \I\partial_t \psi(t) &= H_{\omega} \psi(t)\\
  \psi(0) &= e_0.
\end{split}\ee
It would be somewhat more tedious to consider more general initial
states. For $p \geq 1$, consider the moment operator
\be
 X(p,\omega) = \sup_{t\geq 0} \left( \sum_{n \in \Z^d} (1 + |n|^2)^{p} |\psi(t,n)|^2\right),
\ee
where $|n|^2 = \sum_{k=1}^{d} (n_k)^2$.
We will show

\begin{theorem}\label{thm:dynloc}
 Let $p \geq 1$ then for almost every $\omega$, we have
 \be
  X(p,\omega) < \infty.
 \ee
\end{theorem}

We now begin the proof of this theorem. We will in fact
show that the almost sure set, is the same as in Theorem~\ref{thm:ppsqrt}.
So let $\omega\in\Omega_1$. The first step will be to
rewrite the time evolution \eqref{eq:schroeevol} in terms of
the eigenfunctions of $H_{\lam,\omega}$.

Denote by $E_{\omega}^{\alpha}$ and $\varphi_{\omega}^{\alpha}$ the
orthonormal basis of $\ell^2(\Z^d)$ consisting of eigenfunctions
of $H_{\lam,\omega}$. We have that
$\varphi_{\omega}^{\alpha}(t) = \E^{-\I t E_{\omega}^{\alpha}} \varphi_{\omega}^{\alpha}(0)$
and that
\be
 \psi(t) = \sum_{\alpha} \varphi_{\omega}^{\alpha}(0) \cdot \varphi_{\omega}^{\alpha}(t).
\ee
In particular
\be
 |\psi(t,n)|^2 \leq \sum_{\alpha} |\varphi_{\omega}^{\alpha}(0)|^2 \cdot |\varphi_{\omega}^{\alpha}(n)|^2.
\ee
Hence, it suffices to show that
that
\be
 \sum_{\alpha} \left(|\varphi_{\omega}^{\alpha}(0)|^2 \cdot \left( \sum_{n \in \Z^d} (1 + |n|^2)^{p} |\varphi_{\omega}^{\alpha}(n)|^2\right)\right)
\ee
is finite.

Introduce for $s \geq 0$ the set
\be
 A_{\omega,s} = \{\alpha:\quad \frac{1}{2^{s + 1}} < |\varphi_{\omega}^{\alpha}(0)|^2 \leq\frac{1}{2^{s}}\}.
\ee
Clearly, this set is finite. Furthermore, our task reduces to showing
that for almost every $\omega$ the sequence
\be
 \frac{1}{2^{s}} \sum_{\alpha\in A_{\omega,s}} \left(\sum_{n \in \Z^d} (1 + |n|^2)^{p} |\varphi_{\omega}^{\alpha}(n)|^2\right)
\ee
is summable. 

\begin{proposition}
 There exists $c > 0$. For $\omega\in\Omega_1$, there exists $R_{\omega}$.
 Let $\alpha \in A_{\omega,s}$, then for $|n|_{\infty} \geq \max(4 s^2, R_{\omega})$, we have 
 \be
  |\varphi_{\omega}^{\alpha}(n)| \leq \E^{-c \sqrt{|n|_{\infty}}}.
 \ee
 Furthermore, we have for $R \geq \max(4 s^2, R_{\omega})$ that
 \be
  \sum_{n \in \Lambda_{R}(0)} |\varphi_{\omega}^{\alpha}(n)|^2 \geq \frac{1}{2}.
 \ee
\end{proposition}

\begin{proof}
 The first part is a consequence of (iii)
 of Theorem~\ref{thm:ppsqrt}.
 The second part follows from
 $$
  \sum_{n\in\Lambda_{R}} \E^{-c\sqrt{|n|_{\infty}}} \to 0
 $$
 as $R \to \infty$ and possibly enlarging $R_{\omega}$.
\end{proof}

\begin{lemma}
 Let $\omega\in\Omega_1$, we have
 \be
  \#(A_{\omega,s}) \leq 2\cdot 9^d (\max(4 s^2, R_{\omega}))^{2 d}.
 \ee
\end{lemma}

\begin{proof}
 Denote by $R_{\Lambda_{r_s}(0)}$ the restriction operator
 to $\Lambda_{r_s}(0)$. We have
 $$
  \sum_{\alpha} \|R_{\Lambda_{r_s}(0)} \varphi_{\omega}^{\alpha} \|^2 
   = \| R_{\Lambda_{r_s}}(0)\|_{\mathrm{HS}}^2
   = (\#(\Lambda_{r_s}(0)))^2 \leq (3 r_s)^{2 d}.
 $$
 Next, we have for any $\alpha \in A_{\omega,s}$ that
 $$
  \|R_{\Lambda_{r_s}(0)} \varphi_{\omega}^{\alpha} \|^2 \geq 
  \frac{1}{2} \|\varphi_{\omega}^{\alpha}\|^2.
 $$ 
 In particular, we also obtain
 $$ 
  \#(A_{\omega,s}) = \sum_{\alpha \in A_{\omega,s}} \|\varphi_{\omega}^{\alpha}\|^2
   \leq 2 \sum_{\alpha \in A_{\omega,s}} \|R_{\Lambda_{r_s}(0)} \varphi_{\omega}^{\alpha} \|^2 
    \leq 2 \sum_{\alpha} \|R_{\Lambda_{r_s}(0)} \varphi_{\omega}^{\alpha} \|^2.
 $$
 This implies the claim.
\end{proof}

We are now ready for

\begin{proof}[Proof of Theorem~\ref{thm:dynloc}]
 We first observe that the previous lemmas imply that
 $$
  \sum_{\alpha\in A_{\omega,s}} \left(\sum_{n \in \Z^d} 
   (1 + |n|^2)^{p} |\varphi_{\omega}^{\alpha}(n)|^2  \right) \leq C (4s)^{4 (d + p)}
 $$
 for some $C \geq 1$.
 We also have that
 $$
  \sum_{s \geq 1} \frac{(s)^{4 (d + p)}}{2^{s}} < \infty.
 $$
 The claim follows.
\end{proof}


\section*{Acknowledgements}

I thank Martin Tautenhahn for useful discussions.
I thank Ivan Veseli\'c for his kind invitation for a visit
to the Technische Universit\"at Chemnitz, from when the
argument in Appendix~\ref{sec:spectrum} originates,
and for useful discussions.

\appendix

%
%
%

\section{The initial condition}
\label{sec:initcond}

In this section, I wish to discuss how to obtain the initial scale
estimate at large coupling. A main motivation is that main standard
proofs rely on the Wegner estimate, see \cite{v09} or \cite{kirsch},
which is not available in our context. I will begin by showing that
Hypothesis~\ref{hyp:analytic} implies (ii) of Hypothesis~\ref{hyp:expdecay}.
We will use the notation
\be
 f(\omega) = \sum_{r=0}^{\infty} f(\{\omega_n\}_{n\in\Lambda_r(0)}.
\ee

\begin{lemma}\label{lem:loja}
 Assume Hypothesis~\ref{hyp:analytic}. Then there exist constants
 $F$ and $\alpha > 0$ such that
 \be
  \mathbb{P}(\{\omega\in [-\frac{1}{2},\frac{1}{2}]^{\Z^d}:\quad
   |f(\omega) - E| \leq \eps\}) \leq F\cdot\eps^{\alpha}.
 \ee
\end{lemma}

\begin{proof}
 For $\ti{\omega} = \{\omega_x\}_{x\in\Z^d\setminus\{0\}}$, we
 define
 $$
  g(\ti{\omega}) = \max_{x\in[-\frac{1}{2},\frac{1}{2}]} (f(\{\ti{\omega},x\})
   -\min_{x\in[-\frac{1}{2},\frac{1}{2}]} (f(\{\ti{\omega},x\}).
 $$
 Then by condition (iii) of Hypothesis~\ref{hyp:analytic},
 we have $g(\ti{\omega}) > 0$. Since $g$ is defined on a compact
 space, there exists $\eta > 0$ such that $g(\ti{\omega}) \geq \eta$
 for all choices of $\ti{\omega}$.

 For fixed $\ti{\omega}$, denote by $f_{\ti{\omega}}(x)$ the function
 $f(\{x,\ti{\omega}\})$, so $x$ plays no the role of $\omega_0$. We have
 that $f_{\ti{\omega}}$ obeys the assumption of Cartan's lemma,
 Theorem~\ref{thm:cartanlemma}. So we may conclude that
 $$
  |\{x \in [-\frac{1}{2},\frac{1}{2}]:\quad |f(x)| \leq \E^{-s}\}|
   \leq 30 \E^{3} \exp\left(-\frac{1}{\log(\eta^{-1})} s\right).
 $$
 By Fubini, we thus see that the claim holds with  
 $F = 30\E^3$ and $\alpha = \frac{1}{\log(\eta^{-1})}$.
\end{proof}

We now begin with the proof of the initial condition. The
strategy will be to exhibit a large gap in the spectrum
of $H^{\Lambda_R(0)}$. To conclude the decay of the Green's
function, we will use the Combes--Thomas estimate, whose
consequence, we now recall.

\begin{proposition}\label{prop:combesthomas}
 There exists an universal constant $c_0 > 0$.
 Let $H: \ell^2(\Lambda_r(0)) \to \ell^2(\Lambda_r(0))$ be 
 a Schr\"odinger operator and $\tau\in(0,1)$. Assume that
 \be
  \dist(E, \sigma(H)) \geq \delta
 \ee
 and $r \geq \frac{1}{\delta^2} + 10$.
 Then $\Lambda_r(0)$ is $(c_0 \log(1 + \delta),\tau,3)$-suitable for $H - E$.
\end{proposition}

\begin{proof}
 This is an application of the Combes--Thomas estimate \cite{ct},
 \cite{kirsch}.
\end{proof}

We now exhibit a gap in the spectrum. Here, we use the notation
$(T_x\omega)_n = \omega_{x + n}$. Although this is of little importance
for the proof, $T_x$ could be any measure preserving map.

\begin{lemma}
 Assume there are constants $F > 0$ and $\alpha >0$
 such that for every $E \in \R$
 \be
  \mathbb{P}(\{\omega:\quad |f(\omega) - E| \leq \eps\}) \leq F \eps^{\alpha}.
 \ee
 Let $V_{\lam,\omega}(x) = \lam f(T_x \omega)$.
 Then for $p > 0$, $E \in \R$, $\delta > 0$, and $R \geq 1$,
 there exists $\lam_0 = \lam_0(F,\alpha,p, R, \delta)$
 such that for $\lam > \lam_0$, we have
 \be
  \mathbb{P}(\{\omega:\quad \dist(E, \sigma(H_{\lam,\omega}^{\Lambda_R(0)})) \leq \delta\}) 
   \leq \frac{1}{R^{p}}.
 \ee
\end{lemma}

\begin{proof}
 Introduce the set $\Omega_E$ as the set of $\omega\in\Omega_E$
 satisfying for $x \in \Lambda_R(0)$ that
 $$
  \dist(E, V_{\lam,\omega}(x)) > 2 d + \delta.
 $$
 By assumption, we have that
 $$
  \mathbb{P}(\Omega_E) \leq (3 R)^d F \cdot \left(\frac{2d +\delta}{\lam}\right)^{\alpha},
 $$
 which is $\leq \frac{1}{R^p}$ for $\lam > 0$ large enough.
 Furthermore, one sees that for $\omega\in\Omega_E$, we have
 $$
  \dist(E,\sigma(H_{\lam,\omega})) > \delta,
 $$
 since $\|\Delta\| \leq 2d$. The claim follows.
\end{proof}

Combining this lemma with Proposition~\ref{prop:combesthomas}
and an appropriate choice of $\delta > 0$, we obtain

\begin{theorem}\label{thm:initcond}
 Assume Hypothesis~\ref{hyp:expdecay}.
 For any $\alpha, r_0 < r_1$
 there exists $\lam_0 = \lam_0(\alpha,r_0,r_1,f) > 0$ such that
 \be
  [r_0, r_1]\text{ is $(1,\alpha)$-acceptable for } H_{\lam,\omega} - E
 \ee
 in the sense of Definition~\ref{def:acceptable}.
\end{theorem}

This implies Proposition~\ref{prop:msinitcond}.


%
%
%

\section{The spectrum}\label{sec:spectrum}

The following argument originates from a discussion with Ivan Veseli\'c.

\begin{theorem}
 The spectrum of $H_{\lam,\omega}$ is almost surely an interval.
\end{theorem}

\begin{proof}
 Denote by $\Sigma$ the almost sure spectrum of $H_{\lam,\omega}$.
 Define $(\omega^c)_x = 0$ for $x\in\Z^d$. We then have that
 $$
  \sigma(H_{\lam,\omega^c}) = [-2d + f(\omega^c), 2d + f(\omega^c)].
 $$
 In particular, it is an interval. Let now $\omega \in \Omega_r$
 and define a continuous path $\gamma: [0,1] \to \Omega$ by
 $$
  \gamma(t)_x = t \cdot \omega_x.
 $$
 We then have that $\omega^c = \gamma(0)$ and $\omega = \gamma(1)$.
 We clearly have that
 $$
  \bigcup_{t \in [0,1]} \sigma(H_{\lam,\gamma(t)}) \subseteq \Sigma.
 $$
 Since $\sigma(H_{\lam,\gamma(t)})$ depends continuously
 on $t$, we obtain that this set is an interval, and so
 also $\Sigma$.
\end{proof}


%
%
%

\section{On Wegner's estimate}
\label{sec:wegner}

We now discuss that the conclusions of multi-scale analysis imply
Wegner estimates. This is not new and can for example be found
in \cite{schlag} by Schlag.

\begin{proposition}
 Let $\tau \in (0,1)$ and $\psi(r)$ be a decreasing function
 satisfying $\lim_{r\to\infty}\psi(r) = 0$.
 Assume for all $E$ and $r \geq r_0$ that
 \be
  \mathbb{P}(\Lambda_r(0)\text{ is not $(\gamma,\tau,0)$-suitable for }H_{\omega} - E) < \psi(r).
 \ee
 Then, we have for $R \geq R_0$ that
 \be
  \mathbb{E}\left(\frac{1}{\#\Lambda_R(0)} 
   \tr\left(P_{[E - \eps, E+ \eps]}(H_{\omega}^{\Lambda_R(0)})\right)\right) 
   \leq 7 \psi\left(\frac{1}{3} \left(\log(\eps^{-1})\right)^{\frac{1}{\tau}}\right)^{\frac{1}{d+1}}.
 \ee
\end{proposition}

\begin{proof}
Define
$$
 s = \left\lfloor\frac{1}{3 \cdot (\psi(r))^{\frac{1}{d + 1}}}\right\rfloor.
$$
By Theorem~\ref{thm:resolv1} combined with a trivial
probabilistic estimate, we obtain
$$
 \mathbb{P}( \|(H_{\omega}^{\Lambda_s(0)} - E)^{-1} \| > \E^{3 r^{\tau}} ) < \psi(r)^{\frac{1}{d+1}}.
$$
In particular also
$$
 \mathbb{E}\left(\frac{1}{\#\Lambda_s(0)} 
  \tr\left(P_{[E - \eps, E+ \eps]}(H_{\omega}^{\Lambda_s(0)})\right)\right)
  \leq 2 \psi(r)^{\frac{1}{d+1}} \left(1 +  \eps \E^{3 r^{\tau}}\right).
$$
We choose
$$
 r = \left\lfloor \frac{1}{3} \left(\log(\eps^{-1})\right)^{\frac{1}{\tau}} \right\rfloor.
$$
Hence, we can conclude that for any $R \geq R_0$ that 
$$
 \mathbb{E}\left(\frac{1}{\#\Lambda_R(0)} 
  \tr\left(P_{[E - \eps, E+ \eps]}(H_{\omega}^{\Lambda_R(0)})\right)\right)
  \leq 7 \psi\left(\frac{1}{3} \left(\log(\eps^{-1})\right)^{\frac{1}{\tau}}\right)^{\frac{1}{d+1}}.
$$
This is the claim.
\end{proof}

%
%
%

\end{document}